\documentclass[11pt]{article}
\usepackage{enumerate}
\usepackage{amsmath}
\usepackage{amsthm}
\usepackage{amsfonts}
\usepackage{amssymb}
\usepackage[numbers]{natbib}
\usepackage{xcolor}
\usepackage{graphicx}
\usepackage{tikz}

\usepackage{fullpage}
\usepackage{xcolor}
\usepackage{hyperref}
\usepackage{bbm}
\newcommand*\samethanks[1][\value{footnote}]{\footnotemark[#1]}

\newtheorem{question}{Question}
\numberwithin{question}{section}
\newtheorem{conjecture}[question]{Conjecture}

\newtheorem{corollary}[question]{Corollary}
\newtheorem{theorem}[question]{Theorem}
\newtheorem{proposition}[question]{Proposition}
\newtheorem{lemma}[question]{Lemma}
\newtheorem{claim}[question]{Claim}

\newtheorem{definition}[question]{Definition}

\newtheorem{remark}[question]{Remark}
\numberwithin{equation}{section}

\title{Connectivity for square percolation and coarse cubical rigidity
in random right-angled Coxeter groups} \author{Jason
Behrstock\thanks{Department of Mathematics, Lehman College and The
Graduate Center, CUNY, New York, USA. Email:
\texttt{jason.behrstock@lehman.cuny.edu}.  Research supported as a 
Simons Fellow from the Simons Foundation.}\and R.
Altar {\c{C}}i{\c{c}}eksiz\thanks{Institutionen f\"or Matematik och
Matematisk Statistik, Ume{\aa} Universitet, Sweden.  Emails:
\texttt{altar.ciceksiz, victor.falgas-ravry@umu.se}.  Research
supported by Swedish Research Council grant VR 2021-03687.} \and
Victor Falgas-Ravry\samethanks}

\begin{document}
\maketitle
\begin{abstract}
We consider random right-angled Coxeter groups, $W_{\Gamma}$, 
whose presentation graph $\Gamma$ is taken to be an Erd{\H o}s--R\'enyi random 
graph, i.e., $\Gamma\sim \mathcal{G}_{n,p}$. We use techniques from 
probabilistic combinatorics to establish several new results about the geometry of these random groups.

We resolve a conjecture of Susse and determine the connectivity
threshold for square percolation on the random graph $\Gamma \sim
\mathcal{G}_{n,p}$.  We use this result to determine a large range of $p$ 
for which the random right-angled Coxeter group $W_{\Gamma}$ has a
unique cubical coarse median
structure. Until recent work of Fioravanti, Levcovitz and Sageev, there were no non-hyperbolic examples of groups with cubical
coarse rigidity; our present results show the property is in fact typically satisfied by a random RACG for a wide range of the parameter
$p$, including $p=1/2$.
\end{abstract}

\section{Introduction}

The \emph{right-angled Coxeter group (or RACG)} $W_{\Gamma}$ with
presentation graph $\Gamma=(V,E)$ is the group with generators $V$ and
relations $a^2=\mathrm{id}$ and $ab=ba$ for all $a\in V$ and $ab\in
E$.  We shall investigate RACGs for which the
presentation graph $\Gamma$ is an outcome of the Erd{\H o}s--R\'enyi
random graph model.  This fundamental random graph model is defined as
follows: given $n\in \mathbb{N}$ and a probability $p=p(n)\in [0,1]$,
we define a random graph $\mathcal{G}_{n,p}$ on the vertex set
$[n]:=\{1,2, \ldots n\}$ by including each of the $\binom{n}{2}$
possible edges with probability $p$, independently at random.  We
write $\Gamma\sim \mathcal{G}_{n,p}$ to denote the fact that $\Gamma$
is a random graph with the same distribution as $\mathcal{G}_{n,p}$.
Given such a random graph $\Gamma\sim \mathcal{G}_{n,p}$, we can then
define a random RACG $W_{\Gamma}$, thereby obtaining an interesting
model for a \emph{random right-angled Coxeter group}. Random group models have been extensively studied since the work of Gromov~\cite{Gromov:hyperbolic} in the 1990s, while the geometry of random right-angled Coxeter groups was first considered by Charney and Farber~\cite{CharneyFarber12} in 2012, and has received significant attention, see e.g.\ ~\cite{BehrstockHagenSisto:coxeter,behrstock2018global, Behrstock2024thickness}.

One approach for studying the geometry of a RACG is to investigate the structure of 
its \emph{flats} (isometrically embedded copies of
$\mathbb R^{n}$ in the Davis complex associated to the RACG). 
Using the fact that induced $4$--cycles in $\Gamma$ give rise to two-dimensional 
flats, Dani and Thomas \cite{DaniThomas:divcox} and 
Behrstock, Hagen and Sisto~\cite{BehrstockHagenSisto:coxeter} 
showed that 
important geometric properties of the group $W_{\Gamma}$
could be understood by studying an auxiliary graph encoding the induced $4$--cycles 
of $\Gamma$. The \emph{square graph} of $\Gamma$ is the graph 
$S(\Gamma)$ whose  
vertices correspond to the induced $4$--cycles (or
\emph{squares}) of $\Gamma$, and with edges corresponding to pairs of
induced $4$--cycles having a non-edge in common (we refer to these 
non-edges as \emph{diagonals} of the square).

In probabilistic and combinatorial arguments, it is often convenient to work with a slightly different but closely related auxiliary graph (which, 
confusingly, is also referred to as  ``the square graph'' in the literature), denoted by $T_{1}(\Gamma)$ and encoding 
essentially the same information --- indeed, $S(\Gamma)$ is the line graph of $T_{1}(\Gamma)$;  both $S(\Gamma)$ and $T_{1}(\Gamma)$ are formally defined in Section~\ref{section: preliminaries} below. 
Note that if $\Gamma\sim \mathcal{G}_{n,p}$, then the associated 
square graph (whether $S(\Gamma)$ or $T_1(\Gamma)$) is also a random graph, but with a starkly different and more complex distribution featuring in particular some significant dependencies between the edges.

Studying the square graphs associated to a random graph $\Gamma \sim \mathcal{G}_{n,p}$, is 
interesting both from probabilistic and geometric group theory 
perspectives. Indeed from a probabilistic perspective, these square graphs arising from random graphs present novel challenges, while from a group theoretic perspective they are a source of interesting examples. 
The specific question of the connectivity
threshold of $S(\Gamma)$ when $\Gamma \sim \mathcal{G}_{n,p}$ was raised in the work of
Susse~\cite{susse2023morse} on the existence of Morse subgroups in 
a random RACG, $W_{\Gamma}$. Susse proved that for edge probabilities $p=p(n)$ satisfying  
$n^{-1}\ll  p \leq (1-\varepsilon)\sqrt{(\log 
n)/n}$, the square graph 
$S(\Gamma)$ is a.a.s.\ not connected \cite[Corollary 
3.8]{susse2023morse} (the $n^{-1}\ll  p$ 
hypothesis is needed to ensure  
that $S(\Gamma)$ a.a.s\ has at least two vertices), and he conjectured in~ \cite[Conjecture 
3.9]{susse2023morse} that for $p$ above this range $S(\Gamma)$  is a.a.s.\ connected.

We fully resolve Susse's conjecture and determine the threshold for connectivity of the graphs 
$T_1(\Gamma)$ and $S(\Gamma)$.
\begin{theorem}\label{theorem: connectivity}
	Let $\Gamma\sim \mathcal{G}_{n,p}$ for some function $p=p(n)$ and let $\varepsilon>0$ be fixed. The following hold:
	\begin{enumerate}[(i)]
		\item if $n^{-1}\ll p\leq (1-\epsilon)\sqrt{\frac{\log n}{n}}$, then $S(\Gamma)$ is a.a.s.\ not connected; 
		\item if  $(1+\epsilon)\sqrt{\frac{\log n}{n}}\leq p \leq 1- \omega(n^{-2})$, then $S(\Gamma)$ is a.a.s.\ connected;
		\item if $p\leq \sqrt{2-\epsilon}\sqrt{\frac{\log n}{n}}$, then $T_1(\Gamma)$ is a.a.s.\ not connected; 
		\item if $\sqrt{2+\epsilon}\sqrt{\frac{\log n}{n}}\leq p \leq 1- \omega(n^{-2})$, then $T_1(\Gamma)$ is a.a.s.\ connected.
	\end{enumerate}
\end{theorem}

The thresholds for connectivity of the graphs $S(\Gamma)$ and
$T_1(\Gamma)$ differ.  Indeed, for $p$ in the range
$\left[\sqrt{1+\epsilon}\sqrt{\log n /n},
\sqrt{2-\epsilon}\sqrt{\log n /n}\right]$, $T_1(\Gamma)$
consists a.a.s.\ of a unique non-trivial component together with a
number of isolated vertices, corresponding to non-edges from
$E(\overline{\Gamma})$ that do not lie in any induced square of
$\Gamma$.  So for $p$ in this range a.a.s.\ $T_1(\Gamma)$ is not
connected while its line graph $S(\Gamma)$ is connected.

If one drops from the presentation of a RACG the requirement that  
each generator is of order two, one obtains the presentation for a  
\emph{right-angled Artin group (RAAG)}. These groups interpolate 
between free groups (RAAGs associated to edgeless graphs) 
and free abelian groups (RAAGs associated to complete graphs). 
It was proved by Davis and Januszkiewicz in~\cite{DavisJanuszkiewicz} that every right angled Artin group is a 
finite-index subgroup of some RACG, but it turns out that not every RACG is a 
finite-index subgroup of some right angled Artin group. Several  
interesting papers on determining which RACGs are 
virtually\footnote{Two groups are said to be \emph{virtually isomorphic}, when a finite 
index subgroup of one is isomorphic to a finite index subgroup of the 
other.} RAAGs 
have been written recently, including  
\cite{CashenEdletzberger:visual2dimRAAGRACG, 
DaniLevcovitz:RAAGsubgroupsRACGs}. 
Note that, RAAGs 
either have quadratic divergence\footnote{\emph{Divergence} functions are a key geometric object for the study of geodesic spaces. Roughly speaking, a divergence function measures the length of a shortest path between pairs of points at distance $r$ apart in a space when forced to avoid a ball of radius $\alpha r$ around a third point, for some fixed $\alpha\in (0,1)$, as a function of $r$. Its growth rate (linear, quadratic, ...) has important implications for the geometry of the space. See~\cite{BehrstockDrutu:thick2} and~\cite{DrutuMozesSapir} for formal definitions of divergence.}, or are free abelian (and thus have 
linear divergence), or split as free products (and thus have 
exponential divergence). By~\cite[Theorem 3.2]{BehrstockHagenSisto:coxeter} and~\cite[Theorem~V]{BehrstockHagenSisto:coxeter}
 a RACG is a.a.s.\ quasi-isometric to a 
group that splits or to an abelian group if and only if $p=p(n)$ goes very 
quickly to 0 or 1 (see~\cite{BehrstockHagenSisto:coxeter} for the precise bounds). Hence for $p=p(n)$ bounded away from $0$ and $1$, a random RACG 
if a.a.s.\ quasi-isometric (or, a fortiori, virtually) a RAAG 
only if it has a.a.s.\ quadratic divergence, which by 
\cite[Theorem~1.6]{behrstock2022square} occurs if and only if	
$(\sqrt{\sqrt{6}-2}+\epsilon ) / \sqrt{n} \leq p(n) \leq 1- 
(1+\epsilon) \log{n} / n$. 
Since both RAAGs and RACGs act geometrically on CAT(0) cube complexes with factor 
systems, it follows from 
\cite[Corollary~E]{BehrstockHagenSisto:quasiflats} that 
the structure of their quasiflats can be described in terms of 
certain standard flats in the CAT(0) structure. In the case of RAAGs, 
one can deduce they have quadratic divergence by verifying the 
connectivity associated to the coarse intersection 
pattern of the cosets of their standard $\mathbb Z \times \mathbb Z$ 
subgroups \cite[Theorem~10.5]{BDM}. 
By pulling back the 
subgoups associated to squares in $\Gamma$ to the RAAG, it follows 
that any RACG 
which is quasi-isometric to a RAAG with quadratic divergence would 
inherit this connectivity structure and thus have 
$S(\Gamma)$ connected.

Susse proved in \cite[Corollary~3.7]{susse2023morse} 
(see Theorem~\ref{theorem: connectivity}.(i)),  
that for edge probabilities $p=p(n)$ satisfying $ 1/n\ll p 
<(1+\varepsilon)\sqrt{\log n/n}$  
the random RACG cannot be 
a.a.s.\ quasi-isometric to a RAAG, since $S(\Gamma)$ is a.a.s.\ not connected in that range. By contrast, it is an immediate 
consequence of our Theorem~\ref{theorem: connectivity}.(ii), that the 
same obstruction does not hold when 
$(1+\varepsilon)\sqrt{\log n/n}\leq p \leq 1- \omega(n^{-2})$.  
Accordingly, the following 
remains a very interesting open question: 

\begin{question}
	Given $p=p(n)$ satisfying 
	$(1+\epsilon)\sqrt{\log n/n}\leq p \leq 1- 
	\omega(n^{-2})$ and $\Gamma\sim \mathcal{G}_{n,p}$, is the random 
	right-angled Coxeter group $W_{\Gamma}$ a.a.s.\  virtually a RAAG? If not, is it at least 
	a.a.s.\ quasi-isometric to one?
\end{question}

\medskip

A powerful method in geometric group theory 
involves the study of actions of groups on non-positively curved
spaces. Such actions allow one to deduce important topological,
algebraic, and dynamical results about the group.  Given a
right-angled Coxeter group, its \emph{Davis complex} is a
non-positively curved cube complex on which it acts, see
Davis~\cite{Davis:book} and Gromov~\cite{Gromov:hyperbolic}.  Since
the action on the Davis complex is cocompact, it can be used to associate to any given
right-angled Coxeter group a cocompact cubulation.

Fioravanti, Levcovitz,
Sageev~\cite{FioravantiLevcovitzSageev:cubicalrigidity} recently
established a useful framework for determining the possible
cubulations of a given group.  Their work is phrased in terms of
\emph{medians}, whose definition we now give.  Given a triple of
vertices $x_1,x_2,x_3$ in the $1$--skeleton of a non-positively curved
cube complex, there exists a unique point $m=m(x_1, x_2,x_3)$ in the
$1$--skeleton which lies on the intersection of the geodesics between
$x_i$ and $x_j$, $i \neq j$, $i,j\in [3]$.  One can thus equip every
non-positively curved cube complex with a \emph{median operator}
$(x_1,x_2,x_3)\mapsto m(x_1, x_2, x_3)$ sending a triple of elements
in the $1$-skeleton to their median.

Given a group acting cocompactly on a cube complex, Fioravanti,
Levcovitz and Sageev define the \emph{cubical coarse median} to be the
pull-back of the median operator via any equivariant quasi-isometry.
Two cubulations induce the same \emph{coarse median structure} if
their cubical coarse medians are uniformly bounded distance apart.
This provides a definition of when two co-compact cubulations should
be considered equivalent.  We say that a group has \emph{coarse
cubical rigidity} if all its cubulations induce the same coarse median
structure.  In~\cite{FioravantiLevcovitzSageev:cubicalrigidity} the
authors prove that for each $n\geq 2$ the group $\mathbb{Z}^{n}$ has
countably many distinct cubical coarse medians, and thus in particular
it does not have coarse cubical rigidity.

Using the methods from the proof of Theorem~\ref{theorem:
connectivity}(ii) together with recent results of Fioravanti,
Levcovitz and Sageev, we obtain the following result on the cubical
coarse median structure of the random RACG:

\begin{theorem}[Coarse cubical rigidity]\label{theorem:uniquecoarsemedian}
	Let $\varepsilon>0$ be fixed and $p=p(n)$  satisfy 
	\begin{align*}
		(1+\epsilon)\sqrt{\frac{\log n}{n}}\leq p\leq 1-(1+\varepsilon)\frac{\log n}{n}.
		\end{align*} 
	Then, a.a.s.\ the right-angled Coxeter group $W_{\Gamma}$ with $\Gamma\sim \mathcal{G}_{n,p}$ has a unique cubical coarse median structure.
\end{theorem}

We note that some upper bound on $p=p(n)$ in
Theorem~\ref{theorem:uniquecoarsemedian} is necessary. 
  Indeed for
$1-p=\theta (n^{-2})$ the right-angled Coxeter group $W_{\Gamma}$ with
$\Gamma\sim \mathcal{G}_{n,p}$ is a.a.s.\ virtually Abelian, as shown
by Behrstock, Hagen and Sisto~\cite[Theorem
V]{BehrstockHagenSisto:coxeter}.  For these values of $p$, it thus
follows that $W_{\Gamma}$ is either a finite group, or an infinite
dihedral group (both of which have unique coarse median structures),
or it has more than one distinct cubical coarse median structures by a
result of Fioravanti, Levcovitz and
Sageev~\cite[Proposition~3.11]{FioravantiLevcovitzSageev:cubicalrigidity},
and all three alternatives occur with probability bounded away from
$0$ by~\cite[Theorem V]{BehrstockHagenSisto:coxeter}.

As for the lower bound on $p=p(n)$, results of
Susse~\cite[Corollary~3.8]{susse2023morse} imply that in the case $n^{-1}\ll p
\ll (1-\epsilon)\sqrt{\frac{\log n}{n}}$ the square graph associated 
to $\Gamma\sim
\mathcal{G}_{n,p}$ will a.a.s.\ contain non-bonded induced squares
(see Section~\ref{section: geometry of RACG} for a definition), which
Fioravanti, Levcovitz and Sageev suggest 
should imply the existence of several distinct cubical coarse median
structures (see the discussion
after~\cite[Corollary C]{FioravantiLevcovitzSageev:cubicalrigidity}). 
If their speculation is correct, then this would
indicate the lower bound on $p(n)$ in
Theorem~\ref{theorem:uniquecoarsemedian} is optimal.

The work of Fioravanti, Levcovitz and Sageev gave the first examples
of non-hyperbolic groups with coarse cubical rigidity.
Theorem~\ref{theorem:uniquecoarsemedian}, combined with
\cite[Corollary C]{FioravantiLevcovitzSageev:cubicalrigidity} and
\cite[Theorem 1.1]{Behrstock2024thickness}, implies that coarse
cubical rigidity is in fact very common in RACGs, being a.a.s.\
satisfied by the random RACG $W_{\Gamma}$, $\Gamma\sim
\mathcal{G}_{n,p}$, for $p=p(n)\in [(1+\epsilon)\sqrt{\log n/n},1-
(1+\varepsilon)\log n/n]$, and in particular for $p=1/2$.  It follows
that for almost all $n$-vertex graphs $\Gamma$, the associated RACG
$W_{\Gamma}$ is non-hyperbolic and has coarse cubical rigidity.

\medskip

A number of questions concerning component evolution
and diameter in the square graphs of random graphs were raised in \cite[Section
7]{behrstock2022square}.  In Theorem~\ref{theorem: second
component} and Proposition~\ref{prop: large p connectivity}, we shed
light on several of those questions.  
For instance, as a key step in
the proof of Theorem~\ref{theorem: connectivity}, we establish an a.a.s.\
upper bound on the size of the second largest component in the
auxiliary graph $T_1(\Gamma)$ for $p$ just above the threshold for the
emergence of the giant square component. This is given by the 
following result which resolves the  
questions about component evolution for $p$ in the corresponding range.

\begin{theorem}\label{theorem: second component}
Let $\Gamma\sim \mathcal{G}_{n,p}$. Then there exists a constant $C>0$ such that if 
\[C\sqrt{\frac{\log \log n}{n}}\leq p(n) \leq 1-\omega(n^{-2}),\]
then a.a.s.\ there exists a unique giant component in $T_1(\Gamma)$ of size $\Omega(n^2)$, with all other components having size $O\left(\left(\frac{\log n}{\log \log n}\right)^2\right)$.
\end{theorem}

We conjecture that already at the time of the emergence of a giant
square component, the second largest square component should have
support of only polylogarithmic order.  
\begin{conjecture}\label{conjecture: second component}
Let $\Gamma\sim \mathcal{G}_{n,p}$ and let $\varepsilon>0$ be fixed. Suppose $p=p(n)\geq 	\frac{\sqrt{\sqrt{6}-2} +\varepsilon}{\sqrt{n}}$.
Then a.a.s.\ all connected components of $T_1(\Gamma)$ except the largest one have size $(\log n)^{O(1)}$.	
\end{conjecture}

\subsection{Organization of the paper}
In Section~\ref{section:
preliminaries} we recall some background on graph theory,
algebraic thickness, and a number of probabilistic tools we
shall require for our arguments.
In Section~\ref{section: main} we begin with an outline of the 
proof of Theorem~\ref{theorem: connectivity}. We then proceed 
to the existence of a unique
giant component in square percolation (Section~\ref{section:
uniqueness of the giant}), the connectivity threshold for the random
square graph (Section~\ref{section: connectivity}) and finally the
application of our work to the geometry of the random RACG
(Section~\ref{section: geometry of RACG}).  We end the paper in
Section~\ref{section: concluding remarks}, with some concluding
remarks and open questions related to our work that would be of
interest for further exploration.
 
\section{Preliminaries and definitions}\label{section: preliminaries}

\subsection{Graph theoretic notions and notation}
We recall here some standard graph theoretic notions and notation 
that we will use throughout the paper. We write $[n]:=\{1,2, \ldots, n\}$, and $[m,n]:=\{m,m+1,\ldots,n\}$ for $n,m\in \mathbb{Z}$ and $x_1x_2\ldots x_r$ to denote the $r$--set $\{x_1, x_2, \ldots, x_r\}$. Given a set $S$, we let $S^{(r)}$ denote the collection of all subsets of $S$ of size $r$. A graph is a pair $\Gamma=(V,E)$, where $V=V(\Gamma)$ is a set of vertices and $E=E(\Gamma)$ is a subset of $V^{(2)}$. All graphs considered in this paper are simple graphs, with no loops or multiple edges.

A subgraph of $\Gamma$ is a graph $\Gamma'$ with $V(\Gamma')\subseteq V(\Gamma)$ and $E(\Gamma')\subseteq E(\Gamma)$. We say $\Gamma'$ is the subgraph of $\Gamma$ induced by a set of vertices $S$ if $\Gamma' =(S, S^{(2)}\cap E(\Gamma))$, and denote this fact by writing $\Gamma'=\Gamma[S]$.  The complement of $\Gamma$ is the graph $\overline{\Gamma}:=(V, V^{(2)}\setminus E)$. The neighbourhood of a vertex $x$ in $\Gamma$ is the set $N_{\Gamma}(x):=\{y\in V(\Gamma): \  xy \in E(\Gamma)\}$.

A path of length $\ell\geq 0$ in $\Gamma$ is a sequence of distinct vertices $v_0, v_1, \ldots, v_{\ell}$ with $v_iv_{i+1}\in E(\Gamma)$ for all $i\in [\ell-1]\cup\{0\}$. The vertices $v_0$ and $v_{\ell}$ are called the endpoints of the path. Two vertices are said to be connected in $\Gamma$ if they are the endpoints of some path of finite length. Being connected in $\Gamma$ is an equivalence relation on the vertices of $\Gamma$, whose equivalence classes form the connected components of $\Gamma$.  If there is a unique connected component, then the graph $\Gamma$ is said to be connected. In such a case, the diameter of $\Gamma$ is the least $\ell$ such that any pair of vertices in $\Gamma$ are connected by a path of length at most $\ell$.

A minimally connected subgraph of a connected graph is called a spanning tree. A cycle of length $\ell$ in $\Gamma$ is a sequence of distinct vertices $v_1, v_2, \ldots ,v_{\ell}$ with $v_iv_{i+1}\in E(\Gamma)$ for all $i \in [\ell]$, where $v_{\ell+1}$ is identified with $v_1$. A forest is an acyclic graph, i.e.,  a graph containing no cycle. An independent set in a graph $\Gamma$ is a set of vertices $I\subseteq V(\Gamma)$ such that the subgraph of $\Gamma$ induced by $I$ is the empty graph, that is to say $E(\Gamma)\cap I^{(2)}=\emptyset$. Conversely, a clique in $\Gamma$ is a set of vertices $K\subseteq V(\Gamma)$ such that the induced subgraph is complete, that is to say $ K^{(2)}\subseteq E(\Gamma)$; we refer to $m:=\vert K\vert$ as the order of the clique $K$, and call $K$ an $m$--clique.

\subsection{The square graph}
The square graph $S(\Gamma)$ was introduced by Behrstock, Falgas--Ravry, Hagen and Susse in~\cite{behrstock2018global} as a way of capturing the algebraic notion of thickness of order 1 and the geometric notion of quadratic divergence in graph-theoretic terms, and was denoted $\square(\Gamma)$ in that paper. A related notion for triangle-free graphs having been previously investigated by Dani and Thomas~\cite{DaniThomas:divcox}, and $S(\Gamma)$ has been subsequently investigated by Susse~\cite{susse2023morse} in connection to the existence of Morse subgroups in random right-angled Coxeter groups.

Combinatorially, it turns out to be more convenient to analyse the
closely related graph $T_1(\Gamma)$, which admits $S(\Gamma)$ as its
line graph ($S(\Gamma)$ being, for its part, better suited to some of
the applications to geometric group theory).  This latter graph was
introduced by Behrstock, Falgas--Ravry and Susse
in~\cite{behrstock2022square}, where confusingly it was also referred
to as `the square graph' and denoted by $\square(\Gamma)$; we hope to
resolve this unfortunate ambiguity in this paper by using the notation
$S(\Gamma)$ and $T_1(\Gamma)$.  Generalizations $T_k(\Gamma)$, $k\in
\mathbb{Z}_{\geq 0}$ of $T_1(\Gamma)$ were used by the authors
in~\cite{Behrstock2024thickness} to study higher orders of algebraic
thickness and divergence, with an inductive definition (which is not 
relevant to the present paper).  For our
purposes however we provide an equivalent definition of 
$T_1(\Gamma)$ in terms of the induced squares in $\Gamma$ and their
pairwise interactions.

\begin{definition} [$T_1(\Gamma)$: the square graph]\label{def:square graph} Given a
graph $\Gamma$, the \emph{square graph of $\Gamma$}, denoted by
$T_1(\Gamma)$, is defined as follows.  The vertex set of $T_1(\Gamma)$
is the collection $E(\overline{\Gamma})$ of non-edges of $\Gamma$.
The edges of $T_1(\Gamma)$ consist of those pairs of non-edges $f,
f'\in E(\overline{\Gamma})$ such that the $4$--set of vertices $f\cup f'$  induces a
\emph{square} (i.e.,  a $4$--cycle) in $\Gamma$.

	We refer to connected components of $T_1(\Gamma)$ as \emph{square components} of $\Gamma$.  Given such a component $C$, we define its support to be
	$\mathrm{supp}(C):=\bigcup_{f\in C}f$, the collection of
	vertices in $V(\Gamma)$ belonging to some $f\in C$. Whenever, $\mathrm{supp}(C)=V(\Gamma)$, we say the component has full support. 
\end{definition}
\begin{remark}[The other square graph: $S(\Gamma)$]
	As mentioned above, for certain applications in geometric group theory, it
is natural to work with another auxiliary graph closely related to (but distinct from) the square graph $T_1(\Gamma)$ namely the line graph of $T_1(\Gamma)$, denoted by $S(\Gamma)$.
This latter graph has the induced squares of $\Gamma$ as its vertices, and as its edges those pairs of induced squares having a diagonal in common.  

As noted in~\cite[Remark~1.2]{behrstock2022square} both versions of a square graph carry
essentially the same information, so working with one rather than the other
is primarily a question of what is convenient for the application one
has in mind.  
\end{remark}
\noindent  The study of the properties of $S(\Gamma)$ and $T_1(\Gamma)$ when $\Gamma$ is a random graph is known as \emph{square percolation}, by analogy with the  \emph{clique percolation} model introduced by Der\'enyi, Palla and Vicsek~\cite{DerenyiPallaVicsek} in 2005, which has received extensive attention from researchers in network science.

Connectivity in the square graph is a delicate matter.  In the next
two subsections, we highlight two features of the square graph that
our arguments will have to contend with.  The first is that
one may have two distinct square components with exactly the same
support (something which makes the study of connectivity considerably
more challenging); the second is that connectivity in the square graph
of $\Gamma$ is not monotonic with respect to
the addition or removal of edges in $\Gamma$.  Indeed, adding or
removing an edge to a graph can cause the associated square graph to
change from being connected to being disconnected or on the contrary
from disconnected to connected.

\subsubsection{Coexistence of distinct square components with the same support}
\label{subsection: distinct but same support}
	We show that it is possible for $T_1(\Gamma)$ to have a somewhat
	surprising property: the existence of more than one component with
	full support.  We illustrate this with a concrete example (from
	which many similar ones can be built by varying the relevant
	parameters, including ones with other numbers of components of 
	full support).
	\begin{proposition}
		There exist graphs $\Gamma$ such that $T_1(\Gamma)$ contains
		two distinct components each having full support in
		$\Gamma$.
	\end{proposition}
	\begin{proof}
	The idea in building such an example is
	to begin with a graph with a square component of full support 
	which had many non-edges which are not the `diagonal'  of any 
	induced square. Then, one can add some additional edges which do not 
	make any of the original squares non-induced, but create new induced squares which can be woven together to form a new component of the square graph which is also of full 
	support. Moreover, with sufficient care, one can do this in a way that ensures that 
	each `diagonal' of a new square is not in 
	the link of the `diagonal' of an old square (or vice versa), 
	so that the resulting two square components are distinct, despite both being supported on the entire vertex set.

	We begin with a graph $G'$ (see Figure~1 below) 
	which is the concatenation of a
	sequence of squares along their diagonals together with 
	some extra edges to make some of the squares in the 
	concatenation non-induced. 
	By producing a graph with fewer induced squares, these extra edges will make it easier later on in the construction to 
	ensure that the second square component with full support we shall build does not merge with the square component in $G'$.
	We name the vertices and the edges of the graph $G'$ as follows:
 \begin{align*}
	 V(G')=&\{v_i:\; \text{for} \; i \in [0,21] \};\\
	 E(G')=&\{v_iv_{i+10},v_iv_{i+12}: i \in [9]\} \cup \{v_0v_{12},v_0v_{21},v_{10}v_{11},v_{10}v_{20}\} \\
	 &\bigcup \{v_iv_{i+1}: i \in [0,9] \}\cup \{v_0v_{10}\}\cup \{v_iv_{i+1}: i\in [11,20]\}\cup \{v_{11}v_{21}\}\\
	 &\bigcup  \{v_iv_{i+2}: i \in [0,8] \}\cup \{v_0v_{9},v_1v_{10}\}\cup \{v_iv_{i+2}: i\in [11,19]\}\cup \{v_{11}v_{20},v_{12}v_{21}\}
\end{align*}
The set on the first line in the definition of $E(G')$ consists of the edges that connect the top and bottom
of the ladder and while the sets on the second and third lines connect pairs of indices
that differ by one and two inside the top layer of the ladder and 
similarly in the
bottom layer (the third set  being used to
decreased the number of induced squares).  We draw the graph $G'$ below, the
vertices $v_0,v_{11}$ in the beginning and the end of the figure are
the same vertices drawn twice to make it easier to see the cyclic
symmetry of the graph  thus constructed.

\begin{figure}[!h]\label{fig1}
\centering

\tikzset{every picture/.style={line width=0.75pt}} 

\begin{tikzpicture}[x=0.75pt,y=0.75pt,yscale=-1,xscale=1]
 		\coordinate (A) at (50,92); 
 		 \coordinate (B) at (114,92); 
		\coordinate (C) at (50,139);
		\coordinate (D) at (114,139);
		\coordinate (E) at (178,92);
		\coordinate (F) at (178,139);
		\coordinate (G) at (246,92);
		\coordinate (J) at (246,139);
		\coordinate (K) at (311,92);
		\coordinate (M) at (311,139);
		\coordinate (N) at (376,92);
		\coordinate (O) at (376,139);
		\coordinate (L) at (419,92);
		\coordinate (P) at (419,139);
		\coordinate (U) at (483,92);
		\coordinate (T) at (483,139);
 		 \draw[color=black,fill=black] (A) circle[radius=0.05cm];
 		 \draw[color=black,fill=black] (B) circle[radius=0.05cm];
		  \draw[color=black,fill=black] (C) circle[radius=0.05cm];
 		 \draw[color=black,fill=black] (D) circle[radius=0.05cm];
		  \draw[color=black,fill=black] (E) circle[radius=0.05cm];
 		 \draw[color=black,fill=black] (F) circle[radius=0.05cm];
		  \draw[color=black,fill=black] (G) circle[radius=0.05cm];
		  \draw[color=black,fill=black] (J) circle[radius=0.05cm];
 		 \draw[color=black,fill=black] (K) circle[radius=0.05cm];
		  \draw[color=black,fill=black] (M) circle[radius=0.05cm];
 		 \draw[color=black,fill=black] (N) circle[radius=0.05cm];
		  \draw[color=black,fill=black] (O) circle[radius=0.05cm];
 		 \draw[color=black,fill=black] (L) circle[radius=0.05cm];
		  \draw[color=black,fill=black] (P) circle[radius=0.05cm];
		  \draw[color=black,fill=black] (U) circle[radius=0.05cm];
 		 \draw[color=black,fill=black] (T) circle[radius=0.05cm];
		
\draw    (50,92) -- (114,92) ;
\draw    (114,92) -- (50,139) ;
\draw    (50,139) -- (114,139) ;
\draw    (50,92) -- (114,139) ;
\draw    (114,92) -- (178,92) ;
\draw    (178,92) -- (112,139) ;
\draw    (112,139) -- (178,139) ;
\draw    (114,92) -- (178,139) ;
\draw    (178,92) -- (246,92) ;
\draw    (246,92) -- (178,139) ;
\draw    (178,139) -- (246,139) ;
\draw    (178,92) -- (246,139) ;
\draw    (246,92) -- (311,92) ;
\draw    (311,92) -- (246,139) ;
\draw    (246,139) -- (311,139) ;
\draw    (246,92) -- (311,139) ;
\draw    (311,92) -- (376,92) ;
\draw    (376,92) -- (311,139) ;
\draw    (311,139) -- (376,139) ;
\draw    (311,92) -- (376,139) ;
\draw    (419,92) -- (483,92) ;
\draw    (483,92) -- (419,139) ;
\draw    (419,138) -- (483,139) ;
\draw    (419,92) -- (483,139) ;
\draw    (50,92) .. controls (90,33) and (159,58) .. (178,92) ;
\draw    (114,92) .. controls (154,33) and (227,57) .. (246,92) ;
\draw    (178,92) .. controls (223,32) and (292,57) .. (311,92) ;
\draw    (246,92) .. controls (286,32) and (357,57) .. (376,92) ;
\draw    (52,139) .. controls (88,184) and (150,175) .. (176,139) ;
\draw    (112,139) .. controls (148,184) and (218,174) .. (244,139) ;
\draw    (176,139) .. controls (212,184) and (283,174) .. (309,139) ;
\draw    (244,139) .. controls (280,183) and (348,174) .. (376,139) ;
\draw    (311,92) -- (335.4,72.7) ;
\draw [shift={(337,71.5)}, rotate = 143.13] [color={rgb, 255:red, 0; green, 0; blue, 0 }  ][line width=0.75]    (10.93,-3.29) .. controls (6.95,-1.4) and (3.31,-0.3) .. (0,0) .. controls (3.31,0.3) and (6.95,1.4) .. (10.93,3.29)   ;
\draw    (376,92) -- (400.4,72.7) ;
\draw [shift={(402,71.5)}, rotate = 143.13] [color={rgb, 255:red, 0; green, 0; blue, 0 }  ][line width=0.75]    (10.93,-3.29) .. controls (6.95,-1.4) and (3.31,-0.3) .. (0,0) .. controls (3.31,0.3) and (6.95,1.4) .. (10.93,3.29)   ;
\draw    (419,92) -- (403.44,75.89) ;
\draw [shift={(402,74.5)}, rotate = 44.14] [color={rgb, 255:red, 0; green, 0; blue, 0 }  ][line width=0.75]    (10.93,-3.29) .. controls (6.95,-1.4) and (3.31,-0.3) .. (0,0) .. controls (3.31,0.3) and (6.95,1.4) .. (10.93,3.29)   ;
\draw    (483,92) -- (465.51,75.81) ;
\draw [shift={(464,74.5)}, rotate = 40.97] [color={rgb, 255:red, 0; green, 0; blue, 0 }  ][line width=0.75]    (10.93,-3.29) .. controls (6.95,-1.4) and (3.31,-0.3) .. (0,0) .. controls (3.31,0.3) and (6.95,1.4) .. (10.93,3.29)   ;
\draw    (309,139) -- (333.37,155.34) ;
\draw [shift={(335,156.5)}, rotate = 215.43] [color={rgb, 255:red, 0; green, 0; blue, 0 }  ][line width=0.75]    (10.93,-3.29) .. controls (6.95,-1.4) and (3.31,-0.3) .. (0,0) .. controls (3.31,0.3) and (6.95,1.4) .. (10.93,3.29)   ;
\draw    (374,139) -- (398.37,155.34) ;
\draw [shift={(400,156.5)}, rotate = 215.43] [color={rgb, 255:red, 0; green, 0; blue, 0 }  ][line width=0.75]    (10.93,-3.29) .. controls (6.95,-1.4) and (3.31,-0.3) .. (0,0) .. controls (3.31,0.3) and (6.95,1.4) .. (10.93,3.29)   ;
\draw    (421,139) -- (402.36,158.04) ;
\draw [shift={(401,159.5)}, rotate = 312.93] [color={rgb, 255:red, 0; green, 0; blue, 0 }  ][line width=0.75]    (10.93,-3.29) .. controls (6.95,-1.4) and (3.31,-0.3) .. (0,0) .. controls (3.31,0.3) and (6.95,1.4) .. (10.93,3.29)   ;
\draw    (481,139) -- (462.36,158.04) ;
\draw [shift={(461,159.5)}, rotate = 312.93] [color={rgb, 255:red, 0; green, 0; blue, 0 }  ][line width=0.75]    (10.93,-3.29) .. controls (6.95,-1.4) and (3.31,-0.3) .. (0,0) .. controls (3.31,0.3) and (6.95,1.4) .. (10.93,3.29)   ;

\draw (26,74) node [anchor=north west][inner sep=0.75pt]   [align=left] {$v_0$};
\draw (25,137) node [anchor=north west][inner sep=0.75pt]   [align=left] {$v_{11}$};
\draw (299,63) node [anchor=north west][inner sep=0.75pt]   [align=left] {$v_4$};
\draw (231,147) node [anchor=north west][inner sep=0.75pt]   [align=left] {$v_{14}$};
\draw (244,71) node [anchor=north west][inner sep=0.75pt]   [align=left] {$v_3$};
\draw (164,147) node [anchor=north west][inner sep=0.75pt]   [align=left] {$v_{13}$};
\draw (166,65) node [anchor=north west][inner sep=0.75pt]   [align=left] {$v_2$};
\draw (84,142) node [anchor=north west][inner sep=0.75pt]   [align=left] {$v_{12}$};
\draw (95,69) node [anchor=north west][inner sep=0.75pt]   [align=left] {$v_1$};
\draw (297,149) node [anchor=north west][inner sep=0.75pt]   [align=left] {$v_{15}$};
\draw (379,105) node [anchor=north west][inner sep=0.75pt]   [align=left] {$\ldots$};
\draw (416,70) node [anchor=north west][inner sep=0.75pt]   [align=left] {$v_{10}$};
\draw (415,142) node [anchor=north west][inner sep=0.75pt]   [align=left] {$v_{21}$};
\draw (483,141) node [anchor=north west][inner sep=0.75pt]   [align=left] {$v_{11}$};
\draw (484,72) node [anchor=north west][inner sep=0.75pt]   [align=left] {$v_0$};
\end{tikzpicture}
\caption{The graph $G'$.}
\end{figure}
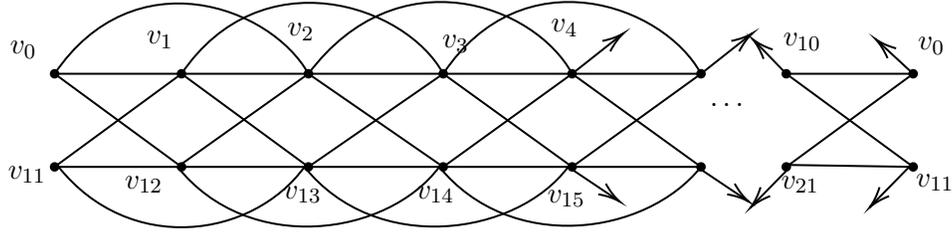

To make the distinction between the graph we will add on top of this,
we note that the diagonals of induced squares in this graph 
(including the blue squares in Figure~3 
are vertices of the non-edges whose indices satisfy the following condition:
\begin{itemize}
	\item one index is $i \in [0,10]$ and the other is  
one of the three indices: $i+11$ or $(i \pm 2) \mod 11 +11$. 
\end{itemize} 

Now, we will add some additional edges to 
construct another square component on these underlying vertices which
will use a completely distinct set of diagonals and which will not
close off any of the diagonals of the above square component.  The idea
is to shift the indices on one side of the above circular strip enough so that when 
we add edges in a similar way to those above there
will be no overlaps of the diagonals.  We shift the bottom row by 6
indices to define $G''$ as follows.

Let $G''$ be the graph with the same set of vertices as $G'$ and whose 
edges,
$E(G'')$, consist of all 
the pairs of indices where the first 
index is $i \in [0,10]$ and the second is  
one of the two indices $(i+6\pm 1) \mod 11 +11$ 
(see  Figure~2). 

\begin{figure}[!h]\label{fig:G''}
    \centering

\tikzset{every picture/.style={line width=0.75pt}} 

\begin{tikzpicture}[x=0.75pt,y=0.75pt,yscale=-1,xscale=1]
		\coordinate (A) at (130,78); 
 		 \coordinate (B) at (130,132); 
		\coordinate (C) at (70,78);
		\coordinate (D) at (70,132);
		\coordinate (E) at (196,78);
		\coordinate (F) at (196,132);
		\coordinate (G) at (264,78);
		\coordinate (J) at (264,132);
		\coordinate (K) at (329,78);
		\coordinate (M) at (329,132);
		\coordinate (N) at (394,78);
		\coordinate (O) at (394,132);
		\coordinate (L) at (439,78);
		\coordinate (P) at (439,132);
		\coordinate (U) at (501,78);
		\coordinate (T) at (501,132);
 		 \draw[color=black,fill=black] (A) circle[radius=0.05cm];
 		 \draw[color=black,fill=black] (B) circle[radius=0.05cm];
		  \draw[color=black,fill=black] (C) circle[radius=0.05cm];
 		 \draw[color=black,fill=black] (D) circle[radius=0.05cm];
		  \draw[color=black,fill=black] (E) circle[radius=0.05cm];
 		 \draw[color=black,fill=black] (F) circle[radius=0.05cm];
		  \draw[color=black,fill=black] (G) circle[radius=0.05cm];
		  \draw[color=black,fill=black] (J) circle[radius=0.05cm];
 		 \draw[color=black,fill=black] (K) circle[radius=0.05cm];
		  \draw[color=black,fill=black] (M) circle[radius=0.05cm];
 		 \draw[color=black,fill=black] (N) circle[radius=0.05cm];
		  \draw[color=black,fill=black] (O) circle[radius=0.05cm];
 		 \draw[color=black,fill=black] (L) circle[radius=0.05cm];
		  \draw[color=black,fill=black] (P) circle[radius=0.05cm];
		  \draw[color=black,fill=black] (U) circle[radius=0.05cm];
 		 \draw[color=black,fill=black] (T) circle[radius=0.05cm];

\draw    (130,78) -- (70,132) ;
\draw    (70,78) -- (130,132) ;
\draw    (196,78) -- (130,132) ;
\draw    (130,78) -- (196,132) ;
\draw    (264,78) -- (196,132) ;
\draw    (196,78) -- (264,132) ;
\draw    (329,78) -- (264,132) ;
\draw    (264,78) -- (329,132) ;
\draw    (394,78) -- (329,132) ;
\draw    (329,78) -- (394,132) ;
\draw    (501,78) -- (439,132) ;
\draw    (439,78) -- (501,132) ;

\draw (65,61) node [anchor=north west][inner sep=0.75pt]   [align=left] {$v_{0}$};
\draw (65,135) node [anchor=north west][inner sep=0.75pt]   [align=left] {$v_{17}$};
\draw (317,61) node [anchor=north west][inner sep=0.75pt]   [align=left] {$v_{4}$};
\draw (255,135) node [anchor=north west][inner sep=0.75pt]   [align=left] {$v_{20}$};
\draw (255,61) node [anchor=north west][inner sep=0.75pt]   [align=left] {$v_{3}$};
\draw (182,135) node [anchor=north west][inner sep=0.75pt]   [align=left] {$v_{19}$};
\draw (182,61) node [anchor=north west][inner sep=0.75pt]   [align=left] {$v_{2}$};
\draw (120,135) node [anchor=north west][inner sep=0.75pt]   [align=left] {$v_{18}$};
\draw (120,61) node [anchor=north west][inner sep=0.75pt]   [align=left] {$v_{1}$};
\draw (317,135) node [anchor=north west][inner sep=0.75pt]   [align=left] {$v_{21}$};
\draw (403,99) node [anchor=north west][inner sep=0.75pt]   [align=left] {$\ldots$};
\draw (433,61) node [anchor=north west][inner sep=0.75pt]   [align=left] {$v_{10}$};
\draw (433,135) node [anchor=north west][inner sep=0.75pt]   [align=left] {$v_{16}$};
\draw (502,135) node [anchor=north west][inner sep=0.75pt]   [align=left] {$v_{17}$};
\draw (502,61) node [anchor=north west][inner sep=0.75pt]   [align=left] {$v_{0}$};

\draw (390,61) node [anchor=north west][inner sep=0.75pt]   [align=left] {$v_5$};

\draw (390,135) node [anchor=north west][inner sep=0.75pt]   [align=left] {$v_{11}$};
\end{tikzpicture}
\caption{ The graph $G''$: note the bottom indices are shifted 
to emphasize the similarity with $G'$.}
\end{figure}
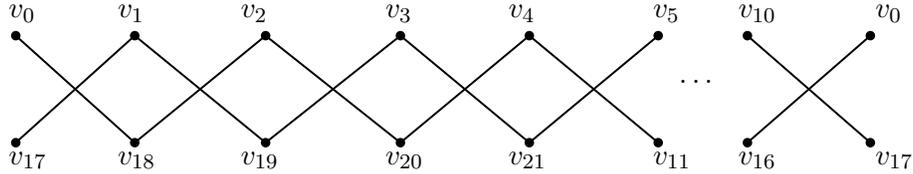

With these two graphs now constructed, we consider the graph $G=G'\cup
G''$, with $V(G)=V(G')=V(G'')$, and $E(G)=E(G')\cup E(G'')$,
illustrated (two ways) in Figure~3 below.

In the graph $G$ the new diagonals of induced squares (i.e., those 
squares which use at least one 
edge from $G''$) consist of all the pairs of indices where:
\begin{itemize}
	\item one  
index is $i \in [0,10]$ and the other is  
one of the three indices: $(i+6) \mod 11 +11$ or $(i+6\pm 2) \mod 11 +11$.
\end{itemize}

\begin{figure}[!h]\label{fig4} 
\centering
\includegraphics[scale=.25]{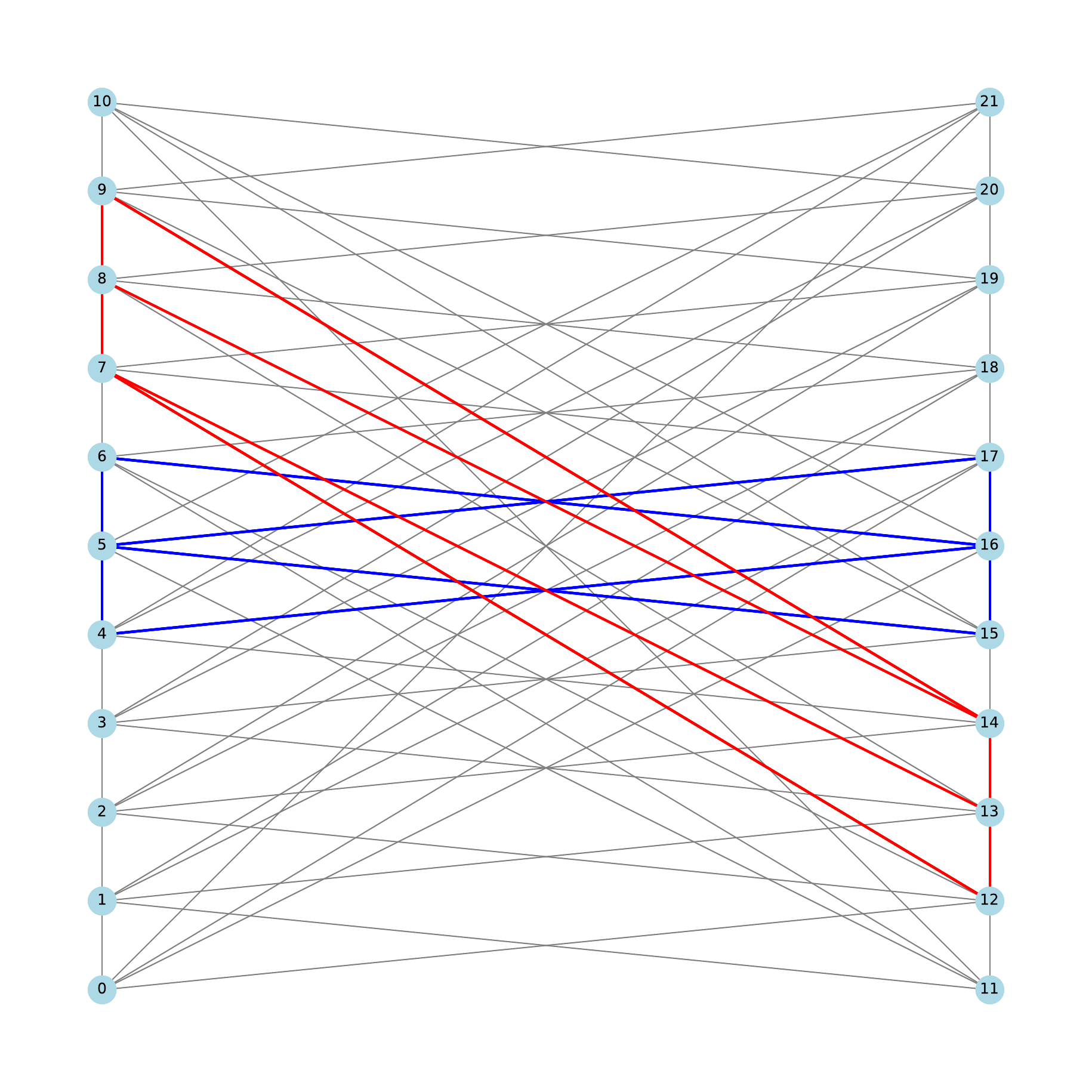}
\includegraphics[scale=.25]{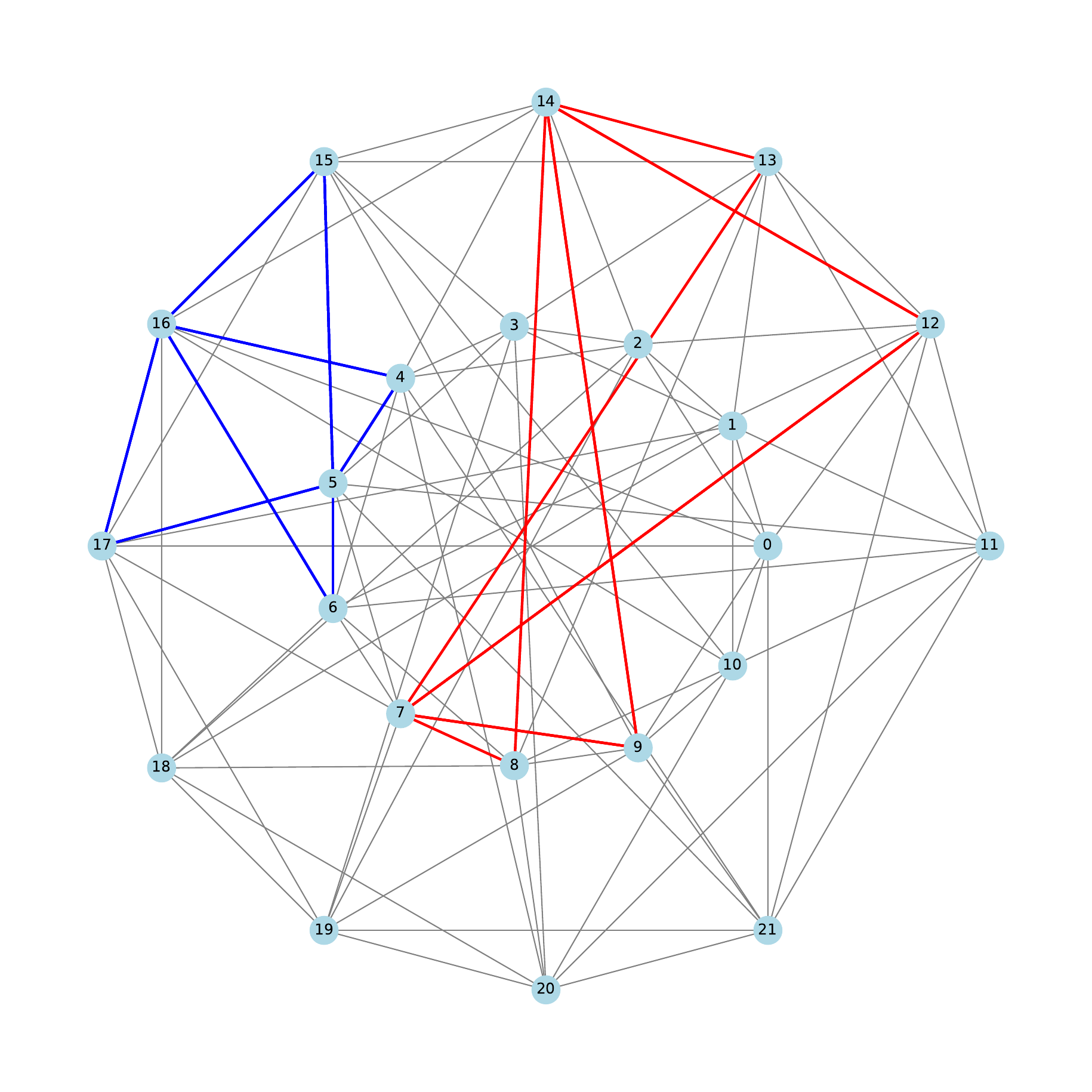}
\caption{Two perspectives on the graph $G$. In both are highlighted  all the squares which contain the vertices 5 and 16 (in blue) and the 
vertices 7 and 14 (in red). Both drawings emphasize the bipartite  division of the vertices, the drawing on the left highlights the 
distinction (as seen by the different ``slopes'') between the  blue squares which are in the first component of full support and the 
red squares which are in the second component of full support. The figure at right is drawn to emphasize the symmetry and to 
avoid the multiple collinear vertices which hide some of the edges in the graph on the left (for instance those connecting the vertices at the bottom to the ones at the top).}
\end{figure}

The diagonals of the first square component with full support are the
diagonals coming from $G'$, and the diagonals of the second square
component with full support are the diagonals coming from squares which
use some edges from $G''$.  Note that both sets of diagonals have full
support by construction. Also note that every diagonal of a square in 
$G$ is of one of the two forms noted in the two bullet points above. 
Since the ones in $G'$ and those including an edge of $G''$ are 
distinct, no diagonal in one can be connected in $T_1(G)$ to a 
diagonal from the other. Thus $T_1(G)$ has two distinct
components both having full support in $G$.
\end{proof}

This construction was independently verified by computer and a number 
of related examples were also generated and 
checked.\footnote{The relevant software was 
developed and implemented in
$\texttt{C++}$ by the first author. Source code and data 
available at: 
\url{http://comet.lehman.cuny.edu/~behrstock} .} 
We note in 
particular that a slight variant of the above ideas yield 
an example in which there are \emph{three} distinct square 
components with full support (using 22 vertices on each side instead 
of 11 and then shifting by 12 instead of 6), 
and further variations also one to build examples with an even larger number of  
distinct square components of full support.

\subsubsection{Non-monotonicity of connectivity in square percolation}
\label{subsection: nonmonotonicity}

The connectivity properties of the graphs $T_k(\Gamma)$ are \emph{not}
monotone under the addition of edges to $\Gamma$.  Indeed, adding an
edge to $\Gamma$ could have a number of different potential effects. 
It could create new induced squares in
$\Gamma$ (and thus new edges in $T_1(\Gamma)$), but it also 
removes a non-edge (so
that $T_1(\Gamma)$ loses a vertex and all edges incident with it), 
which can destroy some 
induced squares that were in $\Gamma$.

As an example, consider the complete bipartite graph $\Gamma=K_{2,n-2}$ with one part of size $2$ and another of size $n-2$. Clearly, $T_1(\Gamma)$ is connected. However if we delete any edge from $\Gamma$, or if we add an edge inside the part of size $2$, then $T_1(\Gamma)$ becomes disconnected (and in the later case becomes an empty graph on $\binom{n-2}{2}$ vertices!). This illustrates the sensitivity and non-monotonicity of the connectivity properties we are considering in this paper.

\subsection{The extremal bound for thick graphs}
We also state here the following useful property of graphs that have a square component with full support, which was first proved in~\cite[Lemma~5.5]{behrstock2018global}, and later greatly generalized so as to encompass all thick graphs in~\cite[Theorem~1.4]{Behrstock2024thickness}.
\begin{proposition}[Extremal bound for thick graphs~\cite{behrstock2018global,Behrstock2024thickness}]\label{prop: combinatorial characterisation of thickness}
	Let $\Gamma$ be a graph on $n$ vertices such that $T_1(\Gamma)$ contains a component $C$ with $\mathrm{supp}(C)= V(\Gamma)$. Then $e(\Gamma)\geq 2n-4$.
\end{proposition}

\subsection{Probabilistic notation and tools}
We write $\mathbb{P}$, $\mathbb{E}$ and $\mathrm{Var}$ for probability, expectation and variance respectively. 
We say that a sequence of events $\mathcal{E}=\mathcal{E}(n)$, $n\in \mathbb{N}$  holds \emph{a.a.s.}\ (asymptotically almost surely) if $\lim_{n\rightarrow \infty} \mathbb{P}(\mathcal{E}(n))=1$. Throughout the paper we use standard Landau notation: for functions $f,g: \ \mathbb{N}\rightarrow \mathbb{R}$ we write $f=o(g)$ if $\lim_{n\rightarrow \infty} f(n)/g(n)=0$, $f=O(G)$ if there exists a real constant $C>0$ such that  $\limsup_{n\rightarrow \infty} \vert f(n)/g(n)\vert \leq C$. We occasionally use $f\ll g$ as a shorthand for $f=o(g)$. We further write $f=\omega(g)$ if $g=o(f)$ and $f=\Omega(g)$ if $g=O(f)$. Finally we write $f=\theta(g)$ if $f=O(g)$ and $f=\Omega(g)$ both hold.

We shall make repeated use of Markov's inequality: given a non-negative integer-valued random variable $X$,  $\mathbb{P}(X>a)\leq \frac{1}{a+1}\mathbb{E}X$ for any integer $a \geq 0$. We shall also use the following standard Chernoff bound, see e.g. ~\cite[Theorems A.1.11 and A.1.13]{AlonSpencer:book}:
\begin{proposition}[Chernoff bound]\label{prop: Chernoff}
Let $X_1, X_2, \ldots, X_n$ be independent and identically distributed Bernoulli random variables. Set $X=\sum_i X_i$ and $\mu:=\mathbb{E}X$. Then for any $\delta\in (0, 2/3)$, 
\[\mathbb{P}(\vert X-\mu\vert\geq \delta \mu)\leq 2\exp\left(-\frac{\delta^2 \mu}{3}\right).\] 	
\end{proposition}
Finally, we shall need Janson's celebrated inequality~\cite{JansonLuczakRucinski90} as well as a large deviation extension of it due to Janson~\cite{Janson90}, which can be found in the form given below in e.g.~\cite[Theorem 8.1.1 and 8.7.2]{AlonSpencer:book}. Given a set $X$ and $q\in [0,1]$, a $q$-random subset $X'$ of $X$ is obtained by including each element of $X$ in $X'$ with probability $q$, independently of all other elements. 
\begin{proposition}[Janson inequalities]\label{proposition: Janson inequality}
	Let $U$ be a finite set and $U_q$ a $q$-random subset of $U$ for some $q\in [0,1]$. Let $\mathcal{F}$ be a family of subsets of $U$, and for every $F\in \mathcal{F}$ let $I_F$ be the indicator function of the event $\{F\subseteq U_q\}$. Set $I_{\mathcal{F}}=\sum_{F\in \mathcal{F}} I_F$, and let $\mu=\mathbb{E} I_{\mathcal{F}}$ and $\Delta=\sum_{F,F' \in \mathcal{F}: \ F\cap F' \neq \emptyset, \  F\neq F'} \mathbb{E} (I_F I_F')$. Then
	\begin{align*} \mathbb{P}\left(I_{\mathcal{F}}=0\right)\leq \exp \left(-\frac{\mu^2}{2\Delta}\right) && \textrm{ and  }&& \mathbb{P}\left(I_{\mathcal{F}}=0\right)\leq \exp \left(-\mu +\frac{\Delta}{2} \right).\end{align*}
	Further, for any $\gamma>0$,
		\begin{align*} \mathbb{P}\left(I_{\mathcal{F}}\leq (1-\gamma)\mu\right)\leq \exp \left(-\frac{\gamma^2\mu}{2 + 2(\Delta/\mu)}\right).\end{align*}
\end{proposition}

\subsection{Building and connecting giants in square percolation}
We shall need two results of Behrstock, Falgas--Ravry and Susse 
from~\cite{behrstock2022square}.  The first establishes 
that for $p$ slightly above the
threshold for the emergence of a giant square component, the number of
non-edges in somewhat-large square components is concentrated around
its mean (which is of order $\Omega(n^2)$).
\begin{proposition}[{Many edges in large components~\cite[Corollary
6.3]{behrstock2022square}}]\label{prop: many edges in large components}
Let $c>\sqrt{\sqrt{6}-2}$ be fixed, and let $C>0$ be an arbitrary constant. Let $\Gamma\sim \mathcal{G}_{n,p}$ and suppose $p=p(n)$ satisfies 
\begin{align*}
	\frac{\lambda}{\sqrt{n}}\leq p(n)\leq  C\sqrt{\frac{\log n}{n}}.
	\end{align*}
Then with probability $1-O(n^{-1})$, the number $N_v$ of non-edges $f\in E(\overline{\Gamma})$ that belong to square components of size at least $(\log n)^4$ satisfies 
\[N_v=(1-o(1))\mathbb{E}N_v =\Omega(n^{2}).\] 
\end{proposition}
\noindent (This was formally stated with $C=5$ in~\cite[Corollary 6.3]{behrstock2022square}, but the proof works for any choice of the constant $C>0$ in the upper bound on $p(n)$.)

The second result concerns vertex-sprinkling; it allows a number of 
somewhat-large square components to connect up into a giant square component.
\begin{proposition}[{Sprinkling lemma~\cite[Lemma 6.6]{behrstock2022square}}]\label{prop: sprinkling lemma}
Suppose we have a bipartition of $V(\Gamma)$ into $V_1\sqcup V_2$ such that:
\begin{enumerate}[(i)]
	\item $\vert V_2\vert =\Omega(n)$;
	\item there are $\Omega(n^2)$ pairs from $E(\overline{\Gamma[V_1]})$ that lie in components of $T_1(\Gamma[V_1])$ of size at least $(\log n)^4$;
	\item $\Omega\left(\frac{1}{\sqrt{n}}\right) \leq p(n) \leq O\left(\left(\frac{\log n}{n}\right)^{\frac{1}{3}}\right)$.
\end{enumerate}
Then the probability that there exists a square component $C$ in $T_1(\Gamma)$ containing all the pairs from $E(\overline{\Gamma[V_1]})$ that lie in components of $T_1(\Gamma[V_1])$ of size at least $(\log n)^4$ is at least
\[1- \exp \left(-\Omega\left((\log n)^2\right)\right).\]
\end{proposition}
\noindent (Formally, this was stated with an upper bound of
$p(n)=O\left(\sqrt{(\log n)/n}\right)$, but the proof of the
connecting lemma~\cite[Lemma 6.4]{behrstock2022square}, which is the
main tool in the proof of~\cite[Lemma 6.6]{behrstock2022square},
carries over when we have the weaker upper bound $p(n) \leq
O\left(\sqrt[3]{(\log n)/n}\right)$ without any change, as does the
remainder of the proof of~\cite[Lemma 6.6]{behrstock2022square}.)

We shall need an additional result that is somewhat reminiscent of the connecting lemma~\cite[Lemma 6.4]{behrstock2022square}, and is proved in an analogous manner via Janson's inequality.
\begin{proposition}[Giant-connecting lemma]\label{prop: connecting lemma}
Let $\Gamma~\sim\mathcal{G}_{n,p}$, with $p=p(n)$ satisfying
\[ \frac{1}{\sqrt{n}} \ll p(n)\leq 100\left(\frac{\log n}{n}\right)^{\frac{1}{3}}.\]
Let $M=M(n)$ be such that $1\ll M\ll p^{-1}$. Suppose $W_1\sqcup W_2$ is a partition of $V(\Gamma)$ such that $\vert W_1\vert = n/M$, and there are $\Omega(n^2)$ non-edges of $W_2$ that lie in the same (giant) component $C_2$ of $T_1(\Gamma[W_2])$. Then the
probability that there exists some component $C_1$ of $T_1(\Gamma[W_1])$ of size at least $M$ such that  $C_1$ and $C_2$ are not a subset of the same component in $T_1(\Gamma)$ is at most
\begin{align*}
\frac{n^2}{M^3}\exp\left(-\Omega\left(Mnp^2\right)\right). 
\end{align*}
Further, if there exist two distinct components $C_2$ and $C_2'$ of $T_1(\Gamma[W_2])$ of size $\Omega(n^2)$, then with probability $1-\exp(-\Omega(n^2p^2))$ they are a subset of the same component in $T_1(\Gamma)$.
\end{proposition}
\begin{proof}
Let $U$ denote the complete bipartite graph with bipartition $W_1\sqcup W_2$. Clearly, the collection of edges of $\Gamma$ from $W_1$ to $W_2$ is distributed as a $p$-random subset of $U$.

Given a component $C_1$ of $T_1(\Gamma[W_1])$ of size at least $M$, there are $\Omega(n^2M)$ pairs $\{x_1x_2, y_1y_2\}$ with $x_1x_2\in C_1$, $y_1y_2\in C_2$. For each such pair, we define  $F_{x_1x_2, y_1y_2}$ to be the event that $x_iy_j\in E(\Gamma)$ for every $i,j\in [2]$, i.e.,  that $x_1x_2y_1y_2$ induces a square of $\Gamma$ joining $C_1$ to $C_2$ in $T_1(\Gamma)$. The probability of this event is $p^4$, and we are thus in a position to apply Janson's inequality.

We have $\mu =\Omega(n^2Mp^4)$. To compute the parameter $\Delta$, note that $F:=F_{x_1x_2, y_1y_2}$ and $F':=F_{x'_1x'_2, y'_1y'_2}$ are independent unless $X:=\vert \{x_1,x_2\}\cap \{x'_1, x'_2\}\vert $ and $Y:=\vert \{y_1,y_2\}\cap \{y'_1, y'_2\}\vert $ are both strictly positive. If $(X,Y)=(2,2)$ then $F=F'$. Further, note that $\mathbb{E}I_FI_{F'}$ is equal to $p^7$ if $(X,Y)=(1,1)$ and to $p^6$ if $(X,Y)=(2,1)$ or $(1,2)$.

Now for $\{x_1x_2, y_1y_2\}$ fixed we have $O(n)$ ways of picking $y'_1y'_2 \in C_2$ such that $ Y=1$ and at most $M$ ways of picking $x'_1x'_2\in C_1$ such that $ X =1$. Going over all possibilities $(X, Y)\in \{(1,1), (2,1), (1,2)\}$, calculating $I_FI_{F'}$ in each case and using our assumption that $M\ll p^{-1}$, we see that
\[\Delta = O\left(\mu p^2n \right).\] 
It then follows from Proposition~\ref{proposition: Janson inequality}, Janson's inequality, together with linearity of expectation and Markov's inequality that the probability that one of the at most $n^2/(2M^3)$ choices of a square component $C_1$ in $T_1(\Gamma[W_1])$ with size at least $M$ that does not join up to $C_2$ in $T_1(\Gamma)$ is at most 
\[\frac{n^2}{2M^3}\exp\left(-\Omega\left(\frac{\mu}{p^2n}\right)\right)= \frac{n^2}{M^3}\exp\left(-\Omega\left(Mp^2n\right)\right).\]

It remains to show the `Further' part of the proposition. Let $\mathcal{B}$ be the event that there are fewer than $\frac{n^2}{4M^2}$ non-edges in $\Gamma[W_1]$. Applying the Chernoff bound (Proposition~\ref{prop: Chernoff}), we have 
$\mathbb{P}(\mathcal{B}) = \exp \left(-\Omega(n^2/M^2)\right)$, which by our bounds on $p$ and $M$ is at most $\exp\left(-\Omega(n^2p^2)\right)$.

Suppose the event $\mathcal{B}$ does not occur. Now, note that there are at most $O(1)$ distinct giant square components in $T_1(\Gamma[W_2])$. For any such pair of giants $C_2,C_2'$ and any $x_1x_2\in C_2$, $y_1y_2\in C_2'$ with $\{x_1,x_2\}\cap \{y_1,y_2\}=\emptyset$ and any non-edge $z_1z_2\in E(\overline{\Gamma[W_1]})$, let $F_{x_1x_2, y_1y_2,z_1z_2}$ be the event that $x_iz_j$ and $y_iz_j$ are edges of $\Gamma$ for every $i,j\in [2]$ (and in particular that $C_2$ and $C_2'$ join up in $T_1(\Gamma)$ through two induced squares both having $z_1z_2$ as a diagonal). Clearly the probability of this event is $p^8$.  Conditional on $\mathcal{B}$ not occurring, there are $\Omega(n^6/M^2)$ such triples $x_1x_2, y_1y_2, z_1z_2$. Here again we can apply Janson's inequality with parameters $\mu= \Omega(p^8n^6/M^2)$ and $\Delta=O\left(\mu n^3p^6/M\right)$ to deduce that the probability that $C_2'$ and $C_2$ fail to connect in $T_1(\Gamma)$ through a $F_{x_1x_2, y_1y_2,z_1z_2}$-event when we reveal the edges from $W_1$ to $W_2$ is at most 
\[\exp\left(-\mu^2/2\Delta\right)=\exp(-\Omega(n^3p^3)),\]
where in the last inequality we have used our assumption that $M\ll p^{-1}$. Since as noted above there can be only $O(1)$ pairs of distinct giants $C_2, C_2'$ in $T_1(\Gamma[W_2])$, it follows from Markov's inequality that the probability any two of them fail to join up in $T_1(\Gamma[W])$ is at most $\mathbb{P}(\mathcal{B}) + O(1)\cdot \exp(-\Omega(n^3p^3)) = \exp(-\Omega(n^2p^2))$, as claimed.
\end{proof}

\section{Connectivity in square percolation}\label{section: main} 
\subsection{Proof outline}
Our proof of Susse's conjecture on the location of the connectivity threshold for $S(\Gamma)$ (and $T_1(\Gamma)$), $\Gamma\sim\mathcal{G}_{n,p}$ is structured as follows.

We first show that for `large' $p$ (which for us means
$p=\Omega\left((\log (n)/n)^{\frac{1}{3}}\right)$, $T_1(\Gamma)$ is
connected (Proposition~\ref{prop: connectivity above n to the -1/3}).
We then concentrate on values of $p=p(n)$ which are not `large', but
which are a little above the threshold for the emergence of a giant
component in $T_1(\Gamma)$ (which occurs at $p\approx
\sqrt{\sqrt{6}-2}/\sqrt{n}$).  For $p$ in this range we prove
Theorem~\ref{theorem: second component}, showing that a.a.s.\ the
giant square component is unique and that all other square components
are small, with size $O((\log n)^2)$.  For this part of the argument,
we rely on a partial converse to Propositions~\ref{prop: combinatorial
characterisation of thickness} (Lemma~\ref{lemma: bound on number of
diag in order 1 component}), the tools from~\cite{behrstock2022square}
we discussed in Section~\ref{section: preliminaries}, and somewhat
intricate partitioning arguments.

Finally, we use Lemma~\ref{lemma: bound on number of diag in order 1 component}, Proposition~\ref{prop: combinatorial characterisation of thickness} and Janson's inequality to rule out the existence of small square components of size $O((\log n)^2)$ for $p$ just above Susse's conjectured connectivity threshold, thereby completing the proof of Theorem~\ref{theorem: connectivity}.
\subsection{One square component to rule them all: proof of Theorem~\ref{theorem: second component}}\label{section: uniqueness of the giant}
We begin with a series of propositions showing that for large $p$, the a.a.s.\ existence of a unique giant component is essentially trivial.
\begin{proposition}\label{prop: large p connectivity}
	Let $\Gamma\sim\mathcal{G}_{n,p}$. If $p=p(n)$ satisfies 
	\[ 100\left(\frac{\log n}{n}\right)^{\frac{1}{4}} \leq p(n)\leq 
	1-\omega(n^{-2}),\]
	then a.a.s.\ $T_1(\Gamma)$ has diameter at most two and in particular is connected.
\end{proposition}
\begin{proof}
	\noindent \textbf{Case 1:} $\mathbf{1-p\geq 100/n}$\textbf{.} Let
	$x_1x_2$ and $y_1y_2$ be distinct (but not necessarily disjoint)
	pairs of vertices from $V(\Gamma)$; the size of the common neighbourhood of
	$\{x_1,x_2,y_1,y_2\}$ in $\Gamma$ stochastically dominates a
	$\mathrm{Binomial}(n-4, p^4)$ random variable.  Applying a
	Chernoff bound (Proposition~\ref{prop: Chernoff}) to bound the
	probability of the pairs having fewer than $p^4n/2$ neighbours in
	common and taking the union bound over all such pairs, we see that
	for $p\geq 100\left(\frac{\log n}{n}\right)^{\frac{1}{4}}$,
	a.a.s.\ all distinct pairs $x_1x_2$, $y_1y_2$ from
	$V(\Gamma)^{(2)}$ have at least $\frac{p^4n}{2}$ neighbours in
	common.

	Further, using the bound $\binom{n}{m}\leq \left(\frac{en}{m}\right)^m$, we have that the expected number of $m$--cliques in $\Gamma$ is
	$\binom{n}{m}p^{\binom{m}{2}}\leq\left( \frac{en}{m}p^{\frac{m-1}{2}}\right)^m$. 
	For $m=\lceil \frac{p^4n}{4}\rceil $ and $p\leq 1-\frac{100}{n}$, the value of this expectation is $o(1)$, whence by Markov's inequality a.a.s.\ $\Gamma$ contains no such clique. In particular, a.a.s.\ any $p^4n/2$--set in $\Gamma$ contains at least one non-edge.

	It follows from the two previous paragraphs that a.a.s.\ for any
	pair of distinct non-edges $x_1x_2$, $y_1y_2$ from
	$V(\Gamma)^{(2)}\setminus E(\Gamma)$ there exists a non-edge
	$z_1z_2\in E(\overline{\Gamma})$ with $z_1, z_2$  lying in the common neighbourhood of $\{x_1,x_2,y_1,y_2\}$, whence $x_1x_2$ and $y_1y_2$ are
	connected by a path of length at most two in $T_1(\Gamma)$.  This 
	establishes the proposition in this case.

	\noindent \textbf{Case 2:} $\mathbf{\omega(n^{-2})\leq 1-p\leq
	100/n}$\textbf{.} By classical results on the
	connectivity threshold for the Erd{\H o}s--R\'enyi random graph model
	$\mathcal{G}_{n, 1-p}$ \cite{ErdosRenyi1}, the complement graph $\overline{\Gamma}\sim
	\mathcal{G}_{n, 1-p}$ of $\Gamma$ is a.a.s.\ not connected.  In
	particular non-edges of $\Gamma$ from distinct components of
	$\overline{\Gamma}$ are joined by a square in $\Gamma$.  Since for
	$\omega(n^{-2})\leq 1-p(n)\leq \frac{100}{n}$ the graph
	$\overline{\Gamma}$ a.a.s.\ contains at least two distinct
	non-trivial components (by e.g., a standard second moment
	argument), it follows that for $p$ in this range $T_1(\Gamma)$
	again has a.a.s.\ diameter at most two.\end{proof}

\begin{proposition}\label{prop: connectivity above n to the -1/3}
	Let $\Gamma\sim\mathcal{G}_{n,p}$, and let $\varepsilon>0$ be
	fixed.  If $p=p(n)$ satisfies
	\[ 100 \left(\frac{\log n}{n}\right)^{\frac{1}{3}}\leq p(n)\leq 100\left(\log (n)/n\right)^{\frac{1}{4}}\] 
	then a.a.s.\ $T_1(\Gamma)$ is connected.
\end{proposition}

\begin{proof}
We argue similarly to Case~1 in Proposition~\ref{prop: large p connectivity}.  Since the distribution of the number of neighbours of a triple of vertices is $\mathrm{Binomial}(n-3, p^3)$, applying a Chernoff bound (Proposition~\ref{prop: Chernoff}) and taking the union over all such triples, we see that for $p\geq 100 \left(\log (n)/n\right)^{\frac{1}{3}}$  the following holds: 
\begin{enumerate}[(a)]
\item a.a.s.\ every triple of vertices in $\Gamma$ has at least $p^3n/2$ common neighbours.  
\end{enumerate}
Further, the expected number of cliques of size $m= \lceil p^3n /2\rceil$ is at most $\left( \frac{ne}{m}p^{\frac{m-1}{2}}\right)^m =o(1)$ for $p \leq 100\left(\log (n)/n\right)^{\frac{1}{4}}$. Thus, we have that
\begin{enumerate}[(b)]
	\item a.a.s.\ every set of $p^3n/2$ vertices in $\Gamma$ must contain at least one non-edge of $\Gamma$. 
\end{enumerate}
It follows from the two facts (a) and (b) that
\begin{enumerate}[(c)]
	\item a.a.s.\  for any pair of intersecting non-edges $xy, yz\in E(\overline{\Gamma})$, $xy, yz$ belong to the
	same square component.  
\end{enumerate}
Now, for $p$ in the range we are considering,  classical  results of Erd{\H o}s and R\'enyi~\cite{ErdosRenyi1} imply that
\begin{enumerate}[(d)]
\item a.a.s.\ $\overline{\Gamma}\sim \mathcal{G}_{n,1-p}$ is connected.
\end{enumerate}
It follows from (c) and (d) that $T_1(\Gamma)$ is a.a.s.\ connected.
\end{proof}

By a Chernoff bound, for $p(n)\leq 1-\omega (n^{-2})$ we have that with probability $1-\exp(-\Omega(n^2(1-p)))=1-o(1)$ that there are $\Omega(n^2(1-p))=\omega(1)$ non-edges in $\Gamma$, each of which is a vertex of $T_1(\Gamma)$. It thus follows from Propositions~\ref{prop: large p connectivity} and~\ref{prop: connectivity above n to the -1/3} that for $p(n)$ between $100 (\log(n)/n)^{\frac{1}{3}}$ and $1-\omega(n^{-2})$, a.a.s.\ there exists a unique component in $T_1(\Gamma)$ of order $\Omega(n^2 (1-p))$, proving Theorem~\ref{theorem: second component} in that range.

In the remainder of this subsection, we turn therefore our attention to the range 
\[p(n)\in R=R(n):= [C\sqrt{\frac{\log \log n}{n}},  100 (\log(n)/n)^{\frac{1}{3}}],\] 
where $C>0$ is  a large constant to be specified later. We first establish a simple proposition whose purpose is to extend upwards the range of $p$ for which we can guarantee that many non-edges are in `somewhat large' square components from that given in Proposition~\ref{prop: many edges in large components}.
\begin{proposition}\label{prop: p large-ish, all square components are large-ish}
	Suppose $p(n)$ satisfies $50\sqrt{\log (n)/n} \leq p \leq 100 \left(\log (n)/n\right)^{\frac{1}{3}}$. Then with probability $1-O(n^{-1})$, every non-edge of $\Gamma$ belongs to a square component of size at least $(\log n)^4$. 	
\end{proposition}
\begin{proof}
The expected number of complete subgraphs on $8$ vertices contained in $\Gamma$ is $\binom{n}{8}p^{\binom{8}{2}}=O(n^8p^{28})$, which for $p\leq n^{-\frac{1}{3}+o(1)}$ is $o(n^{-1})$. Thus by Markov's inequality, the event $\mathcal{B}$ that $\Gamma$ contains a complete subgraph on $8$ vertices has probability at most $o(n^{-1})$.

	Consider a pair of vertices $v,v'$ from $V(\Gamma)$ with $vv'\notin E(\Gamma)$.  By a Chernoff bound (Proposition~\ref{prop: Chernoff}), with probability at least $1-2\exp\left(-(\frac{1}{12}+o(1))p^2n\right)= 1-O(n^{-3})$, the set $X$ of common neighbours of $v$ and $v'$ has size at least $\frac{1}{2}p^2n$ and at most $\frac{3}{2}p^2n$.  If $\mathcal{B}$ does not occur, then $\Gamma[X]$ does not contain any complete graph on $8$ vertices, whence by Tur\'an's theorem~\cite{turan1941external} it contains at most $\left(1-\frac{1}{7}+o(1)\right)\binom{\vert X\vert }{2}$  edges. In particular, there is a set $C=X^{(2)}\cap E(\overline{\Gamma})$ of a least $(\frac{1}{7}-o(1))\binom{p^2n/2}{2}>\frac{1}{64}p^4n^2$ non-edges in $\Gamma[X]$, all of which lie inside the same square component as $vv'$.

\textbf{Case 1: $\mathbf{p\geq 4 \log n/\sqrt{n}}$.} Taking a union bound over the $\binom{n}{2}$ pairs $vv'$ from $V(\Gamma)^{(2)}$, the argument above shows that with probability $1- \mathbb{P}(\mathcal{B})-\binom{n}{2}\cdot O(n^{-3})= 1-O(n^{-1})$, every pair $vv'\in E(\overline{\Gamma})$ is in a component of $T_1(\Gamma)$ of size at least $(\log n)^4$.

\textbf{Case 2: $\mathbf{p\leq 4 \log n/\sqrt{n}}$.} Suppose
$\mathcal{B}$ does not occur.  With $X$ and $C$ as above, and $C$ in
particular of size at least $\frac{1}{64} p^4n^2>10000 (\log n)^2$,
there must exist a perfect matching $e_1, e_2, \ldots, e_m$ of
$m:=\lceil 100 \log n\rceil $ non-edges of $\Gamma$ inside $X$ (i.e., 
$m$ pairwise disjoint non-edges).  We use these non-edges to explore
the square component containing $vv'$.  For $1\leq t\leq m$, do the
following: set $X_t$ to be the set of vertices in $V(\Gamma)\setminus
(X\cup\{v,v'\})$ sending edges to both vertices in $e_t=v_tv'_t$, and
let $C_t:=X_t^{(2)}\cap E(\overline{\Gamma})$.  If $X_t$ has size less
than $\frac{1}{2}p^2n$ or more than $\frac{3}{2}p^2n$, terminate the
process.  Else it follows from Tur\'an's theorem, just as in the
analysis about Case $1$, that $C_t$ has size at least $10000(\log n)^
2$ and contains a collection $e_{(t-1)m+1}, \ldots, e_{tm}$ of $m$
non-edges in $X_t$ forming a perfect matching.  Add $C_t$ to $C$ and
$X_t$ to $X$.  Then if $t\geq (\log n)^2$, terminate the process,
otherwise move to step $t+1$ of our iterative exploration.

We claim that either $\mathcal{B}$ occurs  or  with probability $1-O(n^{-3})$ the process ends with $t=\lceil (\log n)^2\rceil $ and we have $\frac{1}{2}p^2n \leq \vert X_t\vert \leq \frac{3}{2}p^2n$ at each time step $t$: $1\leq t\leq (\log n)^2$. Note that if this is the case then $C$ at the end of the process has size at least $(\log n)^4$, and is part of the same component as $vv'$ in $T_1(\Gamma)$. Taking a union bound over all pairs $vv'$, our claim would thus imply that with probability $1- \mathbb{P}(\mathcal{B})-\binom{n}{2}\cdot O(n^{-3})= 1-O(n^{-1})$, every pair $vv'\in E(\overline{\Gamma})$ is in a component of $T_1(\Gamma)$ of size at least $(\log n)^4$, as desired.

Let us therefore prove our claim. Note first of all that at each time step $t<(\log n)^2$ before the process terminates we have $V(\Gamma)\setminus (X\cup\{v,v'\}) > n - 2t(\log n)^2 p^2n=(1-o(1))n$.  By a Chernoff bound, with probability  $1-2\exp\left(-(\frac{1}{12}+o(1))p^2n\right)= 1-O\left(n^{-4}\right)$,  we have that $X_t\in [\frac{p^2n}{2}, \frac{3p^2n}{2}]$. Thus as claimed either $\mathcal{B}$ occurs or with probability $1- (\log n)^2 O(n^{-4})=1-O(n^{-4})$ the process ends with $t=\lceil (\log n)^2$ and $\frac{1}{2}p^2n \leq \vert X_t\vert \leq \frac{3}{2}p^2n$ holding at each time step $t$: $1\leq t\leq (\log n)^2$.
\end{proof}
\begin{corollary}\label{cor: giant exists}
Fix a subset $V' \subseteq V(\Gamma)$ with $\vert V'\vert =\left(1-o(1)\right)n$, let $p=p(n) \in R$. Then with probability $1-O(n^{-1})$, there exists a giant square component in $T_1(\Gamma[V']) $ of size $\Omega(n^2)$.	
\end{corollary}
\begin{proof}
Fix $\delta>0$, and partition the set $V'$ as $V'=V_1\sqcup V_2$, where $\vert V_2\vert =(\delta +o(1))n$ and $\vert V_1\vert = (1-\delta-o(1))n$. By Propositions~\ref{prop: many edges in large components} and~\ref{prop: p large-ish, all square components are large-ish}, we have that with probability $1-O(n^{-1})$ there are $\Omega(n^{2})$ non-edges of $\Gamma[V_1]$ that lie in square components of $T_1(\Gamma[V_1])$ of size at least $(\log n)^4$. The corollary is then immediate from the sprinkling lemma, Proposition~\ref{prop: sprinkling lemma}.
\end{proof}
\noindent The last tool we need is the following partial converse to Proposition~\ref{prop: combinatorial characterisation of thickness}, giving a lower bound on the number of non-edges in a graph $\Gamma$ that has a square component of full support.
\begin{lemma}\label{lemma: bound on number of diag in order 1 component}
	Suppose $\Gamma$ is a graph on $m\geq 4$ vertices, and that $C$ is a square component of $\Gamma$ with full support, $\mathrm{supp}(C)=V(\Gamma)$. Then $\vert C\vert \geq m/2$.
\end{lemma}

\begin{proof}
	By definition, every vertex $u\in V(\Gamma)=\mathrm{supp}(C)$ belongs to a non-edge $uv\in E(\overline{\Gamma})$ with $uv\in C$.  By double-counting, $C$ must have size at least $m/2$.
\end{proof}

\noindent With these results in hand, we are now ready to prove Theorem~\ref{theorem: second component} for the remaining range of $p$, $p\in R$.
\begin{proof}[Proof of Theorem~\ref{theorem: second component}]
Let $C>0$ be a sufficient large constant, and let $p=p(n)$ satisfy 
\[\frac{C\sqrt{\log \log  n}}{\sqrt{n}}\leq p \leq 100\left(\frac{\log n} {n}\right)^{\frac{1}{3}}.\] 	
Let $\Gamma\sim \mathcal{G}_{n,p}$ and write $V$ for $V(\Gamma)$. Let $M=\lfloor \frac{\log n}{8\log \log n}\rfloor$. Take a balanced partition $V=\sqcup_{i=1}^{2M^2}V_i$ of the vertex set of $\Gamma$ into $2M^2$ parts.	Given a subset $S$ of $[2M^2]$, we denote by $V_S:=\bigcup_{i\in S}V_i$ the union of all parts $V_i$ with $i\in S$.

Given any subset $S\in [2M^2]^ {(2M)}$, let $E_S$ be the event that there is a giant square component $C$ inside $T_1(\Gamma[V\setminus V_S])$ containing $\Omega(n^2)$ non-edges of $\Gamma$, and  further that every square component $C'$ of $T_1(\Gamma[V_S])$ of size at least $M$ joins up to (every such giant) $C$ in $T_1(\Gamma)$. We claim that $E_S$ occurs with very high probability.
\begin{claim}\label{claim: all parts good}
	For every $S\in [2M^2]^ {(2M)}$, we have $\mathbb{P}(\overline{E_S})=o\left({\binom{2M^2}{2M}}^{-1}\right)$.
\end{claim}

\begin{proof}
Since $\vert V\setminus V_S\vert=(1-\frac{1}{M}+o(1))n =(1+o(1))n$, it
follows from Corollary~\ref{cor: giant exists} that with probability
$1-O(n^{-1})$ there exists a giant square component $C$ inside
$T_1(\Gamma[V\setminus V_S])$.  If such a giant exists,
Proposition~\ref{prop: connecting lemma} tells us there is a
probability of at most $n^2M^{-3}\exp\left(-\Omega(Mnp^2)\right)$ that
some component of size $M$ in $T_1(\Gamma[V_S])$ fails to connect up
to it in $T_1(\Gamma)$.  Picking our constant $C$ sufficiently large,
we have that for $p\geq C\sqrt{\log\log (n) /n}$ this probability is
$O(n^{-1})$.  Putting it all together, the probability of $E_S$ is
$1-O(n^{-1})=1- o\left({\binom{2M^2}{2M}}^{-1}\right)$, proving the
lemma. It is here that our choice of $M$ comes in, as it ensures 
that ${2eM}^{4M} =\exp \left(\left(\frac{1}{2}+o(1)\right)\log n\right)$ so that $\binom{2M^2}{M}^2\leq \left(\frac{2eM^2}{M}\right)^{2M}=o(n)$.
\end{proof}

Further, given subsets $S,S'\in [2M^2]^ {(2M)}$ with $\vert S\cap S'\vert=2M-1$, let $E_{S,S'}$ be the event that for every pair of components $C$ in $T_1(\Gamma[V\setminus V_S])$ and $C'$ in $T_1(\Gamma[V\setminus V_{S'}])$ such that $C, C'$ both have size $\Omega(n^2)$, $C$ and $C'$ join up in  $T_1(\Gamma)$. Again, we claim that this event $E_{S, S'}$ occurs with very high probability.

\begin{claim}\label{claim: giants join up}
	For every $S,S'\in [2M^2]^ {(2M)}$ with $\vert S\cap S'\vert=2M-1$,  $\mathbb{P}(\overline{E_{S,S'}})=o\left({\binom{2M^2}{2M}}^{-2}\right)$.
\end{claim}

\begin{proof}
Our condition on $M$ ensures $\vert V_S\cap V_{S'}\vert =(1+o(1))\frac{n}{2M^2}(2M-1)=(1+o(1)) \frac{n}{M}=(1+o(1)) \frac{8n\log \log n}{\log n}$ satisfies $1\ll M\ll p^{-1} $. By Corollary~\ref{cor: giant exists}, with probability $1-O(n^{-1})$ there exists a square component in $\Gamma[V \setminus V_S\cap V_{S'} ]$ of size $\Omega(n^{2})$. We can thus apply the `further' part of Propositions~\ref{prop: connecting lemma}	 to deduce that with probability $1-O(n^{-1})-\exp\left(-\Omega(n^2p^2)\right)= 1-o\left({\binom{2M^2}{2M}}^{-2}\right) $ every pair of square components of size $\Omega(n^2)$ from $\Gamma[V\setminus (V_S\cap V_{S'})]$ join up in $T_1(\Gamma)$.
\end{proof}

We are now ready to complete the proof of  Theorem~\ref{theorem: second component}. Applying Claim~\ref{claim: all parts good} and Markov's inequality, we see that a.a.s.\ the event $E_S$ occurs for every subset $S\in  [2M^2]^ {(2M)}$. Further, by Claim~\ref{claim: giants join up} and Markov's inequality, a.a.s\ the event $E_{S,S'}$ occurs for every pair of  subsets $S, S'\in  [2M^2]^ {(2M)}$ with $\vert S\cap S'\vert = 2M-1$.  We may therefore in the remainder of the proof condition on the a.a.s.\ event $\mathcal{E}:=\left(\bigcap_{S\in [2M^2]^{(2M)}}E_S\right)\cap \left(\bigcap_{S,S'\in [2M^2]^{(2M)}:\  \vert S\cap S'\vert =2M-1} E_{S,S'}\right)$.

Observe that if $C$ is a component in $T_1(\Gamma)$ of size at least $\binom{2M}{2}$, then the support of $C$ has size at least $2M$, and by~\cite[Lemma 5.5]{behrstock2018global}, there must be a $2M$--set $X\subseteq V$ and a subset  $C_X\subseteq C\cap X^{(2)}$ such that $C_X$ is a square component in $T_1(\Gamma[X])$ with full support. By Lemma~\ref{lemma: bound on number of diag in order 1 component}, we must  have that $\vert C_X\vert \geq M$.

Now, $X$ meets at most $2M$ of the parts $V_i$, so there exists a set $S\in [2M^2]^{(2M)}$ such that $X\subseteq V_S$. Since $E_S$ holds, we have that $C_X$ joins up in $T_1(\Gamma)$ to every giant ($\Omega(n^2)$-sized) component $C$ of $T_1(\Gamma[V\setminus V_S])$. Further, since $\bigcap_{S,S'}E_{S,S'}$ holds, we have that for every $S,S'\in [2M^2]^{(2M)}$ and every pair of giant components $C$ in $T_1(\Gamma[V\setminus V_S])$ and $C'$ in $T_1(\Gamma[V\setminus V_{S'}])$, $C$ and $C'$ join up to the same square component in $T_1(\Gamma)$ (since there exists a finite sequence $S_0=S, S_1, S_2, \ldots , S_{\ell}=S'$ with $S_i\in [2M^2]^{(2M)}$, $\vert S_i\cap S_{i+1}\vert=2M-1$ and $E_{S_i, S_{i+1}}$ occurring for every $i\geq 0$, and since by $E_{S_i}$ each of the $T_1(\Gamma[V\setminus V_{S_i}])$ contains a giant component).

It follows that there exists a unique  component in $T_1(\Gamma)$ of size greater than $\binom{2M}{2}\geq \frac{1}{33}\left(\frac{\log n}{\log \log n}\right)^2$, and this component has size $\Omega(n^2)$ (since  $E_S$ holds), proving the theorem.
\end{proof}

\subsection{The connectivity threshold in square percolation}\label{section: connectivity}
\begin{proof}[Proof of Theorem~\ref{theorem: connectivity}]
\textbf{Part (i)} follows from the work of Susse~\cite[Corollary
3.8]{susse2023morse}, who proved that for $n^{-1}\ll p\leq
(1-\varepsilon)\sqrt{\log n /n}$, a.a.s.\ $T_1(\Gamma)$ contains
isolated squares: induced copies of the $4$--cycle $C_4$ not sharing a
diagonal with any other induced copy of $C_4$. Each of these isolated 
squares 
corresponds to an isolated edge in $T_1(\Gamma)$, and thus an isolated
vertex in $S(\Gamma)$. The a.a.s\ existence of these  isolated vertices implies that
$S(\Gamma)$ is a.a.s.\ not connected for $p$ in this range.

\textbf{Part (iii)} is proved using a similar second-moment argument.
Suppose $p\leq \sqrt{2-\varepsilon}\sqrt{\log n /n}$.  Then we claim
that a.a.s.\ there are non-edges of $\Gamma$ that are not contained in
any induced square.  Such non-edges correspond to isolated vertices of
$T_1(\Gamma)$, and ensure the latter graph is not connected.  It is
therefore enough to prove our claim, which we now do.

Given a pair $e\in V(\Gamma)^{(2)}$, let $X_e$ be the indicator function of the event that $e$ is a non-edge of $\Gamma$ whose endpoints have no neighbour in common. Observe that if $X_e=1$, then certainly $e$ is not contained in any induced square of $\Gamma$. Set $X=\sum_{e}X_e$. Straightforward calculations then give
\begin{align*}
\mathbb{E}X= \binom{n}{2}(1-p)(1-p^2)^{n-2}=\left(\frac{1}{2}+o(1)\right)n^2e^{-p^2n}\geq \left(\frac{1}{2}+o(1)\right)n^{\varepsilon},	
\end{align*}
 while 
\begin{align*}
\mathbb{E}(X^2)
&=\binom{n}{2}(1-p)^2\left((1-p^2)^{n-2} + 2(n-2)(1-2p^2+p^3)^{n-3} +\binom{n-2}{2}(1-p^2)^{2(n-4)}(1-p^4) \right)\\
&= \left(\mathbb{E}X\right)^2(1+o(1)).	
\end{align*}
Thus $\mathrm{Var}(X)=o\left(\left(\mathbb{E}X\right)^2 \right)$, and
it follows immediately from Chebyshev's inequality that a.a.s.
$X=(1+o(1))\mathbb{E}X\gg 1$.  Thus a.a.s.\ for $p \leq
\sqrt{2-\varepsilon}\sqrt{\log n /n}$, $T_1(\Gamma)$ contains isolated
vertices and is not connected.

\textbf{For Part (ii)}, consider $p=p(n)\geq (1+\varepsilon)\sqrt{\log n}/\sqrt{n}$. By Proposition~\ref{prop: connectivity above n to the -1/3}, we may assume that $p=p(n)\leq 100\left(\log (n)/n\right)^{\frac{1}{3}}$. Further, by Theorem~\ref{theorem: second component}, for $p$ in this range we know that a.a.s.\  there exists a unique square component of size at least $(\log n)^2$.  Thus all that remains to be shown is that a.a.s.\ $T_1(\Gamma)$ contains no smaller non-trivial component, i.e.,  that $T_1(\Gamma)$ consists of a giant component together with a collection of isolated vertices (corresponding to non-edges of $\Gamma$ not belonging to  any induced square of $\Gamma$).

For $k\geq 4$ and a $k$--set $S\subseteq V(\Gamma)$, let $X_S$ be the indicator function of the event that there exists a connected component $C$ of $T_1(\Gamma)$ with support $\mathrm{supp}(C)= S$.  Note that every non-trivial connected component of $T_1(\Gamma)$ must have support of size at least $4$, by definition.  Clearly for $X_S$ to be non-zero, we must have that  $\Gamma[S]$ contains at least $2k-4$ edges (by Proposition~\ref{prop: combinatorial characterisation of thickness}).

Further, suppose $C$ is a square component with support exactly $S$. Then for every non-edge $f=\{x_1,x_2\}\in \overline{\Gamma}[S]$ belonging to $C$, every neighbour $f'=\{y_1,y_2\}$ of $f$ in $T_1(\Gamma)$ and every vertex $v\in V(\Gamma)\setminus S$ sending edges to both of the endpoints of $f$, $v$ must also send edges to both of the endpoints of $f'$. Indeed, suppose $vx_1, vx_2$ are edges in $\Gamma$ but $vy_i$ is not. Then the $4$--set $\{vx_1x_2y_i\}$ induces a square in $\Gamma$ connecting $f$ to $vy_i$ in $T_1(\Gamma)$, contradicting  the fact that $v\notin S=\mathrm{supp}(C)$. Since $X$ has full support on $S$, this implies that any vertex $v\in V(\Gamma)\setminus S$ sending edges to both endpoints of $f$ must in fact send edges to every vertex of $S$ (indeed for any $z\in S$, there is a sequence of non-edges $f_0=f$, $F_1, f_2, \ldots , f_t$ such that for each $i$ $f_i\cup f_{i+1}$ induces a square in $\Gamma[S]$ and $z\in f_t$).

By Lemma~\ref{lemma: bound on number of diag in order 1 component}, any square component $C$ with $\mathrm{supp}(C)=S$ must contain a set $L=\{f_1, f_2$, $\ldots f_{\lceil k/2\rceil}\}$ of at least $k/2$ non-edges in $\Gamma[S]$. Thus the discussion from the previous two paragraphs tells us that if $X_S>0$ then all of the following must hold:
\begin{enumerate}[(a)]
	\item $\vert S\vert =k\geq 4$;
	\item $\Gamma[S]$ contains a least $2k-4$ edges;
	\item there is a set $L\subseteq S^{(2)}$ of $\lceil k/2\rceil$ distinct pairs of vertices from $S$ such that for every vertex $v\in V(\Gamma)\setminus S$, either $v$ sends edges to all of $S$ or for every $f\in L$, at least one of the endpoints of $f$ does not receive an edge from $v$.
\end{enumerate}
We shall use this information to bound the probability that $X_S>0$. Fix a $k$--set $S$, and an arbitrary $\lceil k/2\rceil $--set $L\subseteq S^{(2)}$. Observe that for any $v\in V(\Gamma)\setminus S$, the events $\left(E_{v, f}\right)_{f\in L}$ that $v$ sends edges to both endpoints of $f\in L$ are positively correlated $\mathrm{Bernoulli}(p^2)$ random variable. Further for $f_1\neq f_2$, the probability of $E_{v,f_1}\cap E_{v, f_2}$ is $p^3$ if $f_1\cap f_2\neq \emptyset$ and $p^4$ otherwise. We may thus apply Janson's inequality with $\mu= \vert L\vert p^4$ and $\Delta \leq 2\binom{\vert L}{2}p^3$ to deduce that
\begin{align}\label{eq: jansonbound}
\mathbb{P}\left(\bigcap_{f\in L} \overline{E_{v,f}}\right)\leq e^{- \frac{k}{2} p^2+ \frac{k^2}{8}p^3}.
\end{align}
Now write $\mathcal{A}_{v,L}$ for the event that either (i) $\left(\bigcap_{f\in L} \overline{E_{v,f}}\right)$ occurs or (ii) $v$ sends an edge to every vertex in $S$.  Note that for any choice of $L$, the events $\left(\mathcal{A}_{v,L}\right)_{v\in V(\Gamma)\setminus S}$ are mutually independent, and are independent of the state of edges inside $\Gamma[S]$. Noting for  $p$ in our range and $k\geq 4$ we have $p^ke^{\frac{k}{2}p^2} \leq p^4e^{2p^2}=o(n^{-1})$, it follows from this independence, inequality~\eqref{eq: jansonbound}, conditions (a)--(c) above, our bounds on $p$ and standard bounds for binomial coefficients that for any $k$ with $4\leq k\leq (\log n)^2$ and any $k$--set $S$,
\begin{align*}
\mathbb{P}(X_S=1)&\leq \mathbb{P}\left(\vert E(\Gamma[S])\vert\geq 2k-4\right) \left(\sum_{L\subseteq S^{(2)}: \ \vert L\vert = \lceil k/2\rceil } \quad \prod_{v\in V(\Gamma)\setminus S} \mathbb{P}\left(\mathcal{A}_{v,L}\right)\right)\\
&\leq \binom{\binom{k}{2}}{2k-4} p^{2k-4}\binom{\binom{k}{2}}{\lceil k/2\rceil } \left(p^ k+ e^{-\frac{k}{2}p^2+ \frac{k^2}{8}p^3}\right)^{n-k} \leq(1+o(1)) (ek)^{\frac{5k}{2}}p^{2k-4} e^{-\frac{k}{2}np^2(1-o(1))}.
\end{align*} 	
Setting $X_k:=\sum_{S\in V(\Gamma)^{(k)}} X_S$ and using linearity of expectation, we can exploit the inequality above to bound the expected number of number of square components of $T_1(\Gamma)$ with support of size $k$. Given $k$ with $4\leq k\leq (\log n)^2$ and a fixed $k$--set $S$, we have
\begin{align*}\label{eq: bound on Xk}
	\mathbb{E}(X_k) \leq \binom{n}{k} \mathbb{E}(X_S)< \left(\frac{en}{k} (ek)^{\frac{5}{2}}p^2 e^{-\frac{p^2 n(1+o(1))}{2}}\right)^k p^{-4}\leq o\left(\frac{(\log n)^{5k-2 }}{n^{\frac{(k-4)+k\varepsilon}{2}}}\right),
\end{align*} 	
with the first inequality coming from elementary bound $\binom{n}{k}\leq \left(\frac{en}{k}\right)^k$ on the binomial coefficient $\binom{n}{k}$, and the second one from the monotonicity of $x\mapsto xe^{-x}$ in $[1,+\infty)$, our upper bound on $k$ (namely $k\leq (\log n)^2$), and our lower bound on $p$ (namely$p\geq (1+\varepsilon)\sqrt{\log n}/\sqrt{n}$). Applying Markov's inequality and linearity of expectation, this implies that
\begin{align*}
	\mathbb{P}\left(\sum_{k=4}^{(\log n)^2} X_k>0\right)\leq \sum_{k=4}^{(\log n)^2}\mathbb{E}X_k =O\left(\frac{(\log n)^{20 }}{n^{4\varepsilon}}\right)=o(1),
\end{align*}
and thus a.a.s.\ there is no square component with support of size between $4$ (the smallest possible size) and $(\log n)^2$. By Theorem~\ref{theorem: second component} we know that a.a.s.\  $X_n=1$ and $X_k=0$ for $k\in [(\log n)^2, n-1]$. Thus there is an a.a.s.\ unique non-trivial square component, and $S(\Gamma)$ is a.a.s.\ connected as required.

\textbf{Part (iv)} follows directly from Part (iii) and Markov's inequality: for $p=p(n)$ satisfying $100\left( \log (n)/n\right)^{\frac{1}{3}} \leq p\leq 1-\omega(n^{-2})$, we showed in Proposition~\ref{prop: connectivity above n to the -1/3} that $T_1(\Gamma)$ is a.a.s.\ connected. Further by Part (iii) established above, we know that for  $\sqrt{2+\varepsilon}\sqrt{\log (n)/n}\leq p(n)\leq 100\left(\log (n)/n\right)^{\frac{1}{3}}$, there is an a.a.s.\ unique non-trivial component in $T_1(\Gamma)$. Thus all we need to rule out is the presence of isolated vertices for $p$ in that range, i.e., the existence of `bad' non-edges of $\Gamma$ that are not contained in any induced square of $\Gamma$. Given $e\in V(\Gamma)^{(2)}$, let $X_e$ denote the indicator function of the event that $e$ is a non-edge and that the joint neighbourhood of the endpoints of $e$ forms a clique. Set $X=\sum_e X_e$, so that $X$ is precisely the number of `bad' non-edges of $\Gamma$. Then by Markov's inequality and linearity of expectation, our bounds on $p$  as well as the facts that $x\mapsto (1+x+x^2+x^3)e^{-x}$ is monotonous decreasing when $x\geq 3$ and that $np^{\frac{9}{2}} =o((\log n)^2 n^{-\frac{1}{2}})$, we have that
\begin{align*}
\mathbb{P}\left(X>0\right)\leq \mathbb{E}X & \leq \binom{n}{2}(1-p)\sum_{i=0}^{n-2}\binom{n-2}{i}p^{2i}(1-p^2)^{n-2-i}p^{\binom{i}{2}}\\
&< n^2 \left((1+np^2 +n^2p^5+n^3p^9+n^4p^{14}+n^5p^{20})e^{-p^2(n-7)} +\sum_{i\geq 6} \left(np^{\frac{i+3}{2}}\right)^i \right)\\
&=O\left(n^{-\varepsilon}\log n\right)=o(1). 
\end{align*}
\noindent Thus a.a.s.\ $X=0$ and $T_1(\Gamma)$ is connected.
\end{proof}

\subsection{Implications for the geometry of the random RACG}\label{section: geometry of RACG}
Let $\Gamma$ be any graph.  Fioravanti, Levcovitz and
Sageev~\cite{FioravantiLevcovitzSageev:cubicalrigidity} introduced the
following notion of a \emph{bonded} square: an induced square $X$ in
$\Gamma$ is bonded if there exists an induced square $X'$ in $\Gamma$
such that $X$ and $X'$ have exactly three vertices in common.  In
other words, an induced square with diagonals $f$ and $f'$ is bonded
if and only if there exists some vertex $v\notin f\cup f'$ sending
edges to both endpoints of exactly one of $f$ and $f'$ (and thus
sending at least one non-edge to the other diagonal).

As one of their main results, Fioravanti, Levcovitz and Sageev
established
in~\cite[Corollary~C]{FioravantiLevcovitzSageev:cubicalrigidity} that
if every square in a graph $\Gamma$ is bonded, then the associated
RACG $W_{\Gamma}$ satisfies coarse cubical rigidity, and every
cocompact cubulation of $W_{\Gamma}$ induces the same coarse median
structure on $W_{\Gamma}$ as the one given by its Davis complex.  They further
suggested that this should be an if and only if, i.e., that the
existence of non-bonded squares implies non-uniqueness of the coarse
median structure on $W_{\Gamma}$, though they were unable to prove
this.

We now show how a slight modification of our proof of
Theorem~\ref{theorem: connectivity}(ii) allows us to obtain
a threshold for the disappearance of non-bonded squares in $\Gamma
\sim \mathcal{G}_{n,p}$ and thereby for coarse cubical rigidity.

\begin{proof}[Proof of Theorem~\ref{theorem:uniquecoarsemedian}]
Let $\Gamma\sim \mathcal{G}(n,p)$.  Given a $4$--set $S$ in
$V(\Gamma)$, let $X_S$ be the indicator function of the event that $S$
induces a non-bonded square in $\Gamma$.  For this event to hold, 
first, 
$\Gamma[S]$ must induce a square, which occurs with probability
$3p^4(1-p)^2$.  Secondly, for every vertex $v\notin S$, if $v$ sends
edges to both ends of a diagonal of $\Gamma[S]$ then it must in fact
send edges to all four vertices in $S$; this occurs with probability
$1-2p^2(1-p^2)$, independently at random for each $v\in
V(\Gamma)\setminus S$.  In particular, we have
\begin{align*}
\mathbb{E}X_S=3p^4(1-p)^2 (1-2p^2(1-p^2))^{n-4}\leq 3n^{-2}  \left(np^2(1-p)\right)^2 \exp\left(-2(1+p) \left(np^2(1-p) \right) +O(1)\right). 
\end{align*}
Fix $\varepsilon>0$.  For $p=p(n)$ satisfying the bounds $(1+\varepsilon)\frac{\sqrt{\log n}}{\sqrt{n}}\leq p \leq 1-(1+\varepsilon)\frac{\log n}{n}$, we have that $np^2(1-p) \geq (1+\varepsilon +o(1))\log n$. It then follows from the monotonicity of $x\mapsto x^2 \exp(-x)$ for $x\in \mathbb{R}_{\geq 2}$ and the inequality above that $\mathbb{E}X_S\leq  O\left(n^{-4-2\varepsilon +o(1)}\right)=o(n^{-4})$.

Taking a union bound over all $4$--sets $S$ in $V(\Gamma)$ and 
applying Markov's inequality, we deduce that for $p(n)\in [(1+\varepsilon)\frac{\sqrt{\log n}}{\sqrt{n}}, 1-(1+\varepsilon)\frac{\log n}{n}]$ we have that a.a.s.\ $\sum_S X_S=0$ holds, i.e.,  that a.a.s.\ every square in $\Gamma$ is bonded. Together with~\cite[Corollary C]{FioravantiLevcovitzSageev:cubicalrigidity}, this implies the desired conclusion regarding the coarse cubical rigidity of $W_{\Gamma}$.
\end{proof}

\section{Concluding remarks}\label{section: concluding remarks}

From the point of view of geometric group theory, this paper raises a 
number of questions, some probabilistic in nature and others not. 

\medskip

The graphs constructed in Section~\ref{subsection: distinct but same 
support} came as somewhat of a surprise to the authors. These are very concrete 
examples, but the phenomenon of having multiple distinct square components 
each of full support is new, and one suspects that the RACG associated to these graphs may 
have other interesting geometric properties. We believe these examples warrant further investigation in the future.

As a juxtaposition to Theorem~\ref{theorem:uniquecoarsemedian}, we note that 
Mangione recently showed that there are many right-angled Coxeter 
groups which admit uncountably many coarse median structures  
\cite{Mangione:shortHHGI}. His examples are beasts of a different 
form from the coarse cubical medians at play in our theorem, since they need not be cubical in nature and so in general can behave in a much 
wilder fashion. 
Moreover, while the RACGs he shows this for do not overlap with 
our examples, there is no a priori reason that a group with a 
unique cubical coarse median structure could not have uncountably many 
coarse median structures.  Accordingly, it would be interesting to 
know whether 
his techniques (or others) could be adapted to answer the following question, 
for which we do not even know whether the answer is finite, countably infinite or uncountable!

\begin{question}\label{question: coarse medians}
The random RACG for $p$ in the range given in Theorem~\ref{theorem:uniquecoarsemedian} a.a.s.\ has a unique cubical coarse 
median structure. How many coarse median structures can it possess? Is this cardinality concentrated, and how does its distribution change as $p$ varies? 
\end{question}
\noindent A related question which is left open by our work is:
\begin{question}\label{question: RAAG coarse medians}
For which $p$ does the random RAAG  a.a.s.\ have a unique cubical coarse 
median structure? For $p$ outside that range, how many cubical coarse median structures can it possess? Is this cardinality concentrated, and how does its distribution change as $p$ varies? 
\end{question}
As we were putting the finishing touches to this paper Fioravanti and Sisto~\cite{FioravantiSisto25} 
released a preprint in which they showed that some groups with a mix of hyperbolic/Euclidean 
geometry different than that considered here, also have a unique coarse median (e.g., products of higher 
rank free groups), further highlighting the potential interest of the two questions above.

\medskip

From a probabilistic perspective, there are also a  number of natural problems regarding the connectivity properties of $T_1(\Gamma)$, $\Gamma\sim
\mathcal{G}_{n,p}$. Foremost among them is the question of whether a.a.s.\ as soon as a giant square component emerges, all other components have logarithmic size. We stated this more formally as Conjecture~\ref{conjecture: second component}. Extending the investigation into the connectivity properties of $T_1(\Gamma)$, it is natural to ask how the diameter of that random
graph behaves. The most basic question in this direction is:
\begin{question}\label{question: diameter}
	What is the typical behaviour of the diameter of $T_1(\Gamma)$ when $\sqrt{2\log(n)/n}\leq p\leq 1-\omega(n^{-2})$?
\end{question}
\noindent Heuristically, for $p=o(1)$ in this range one expects the number of vertices of $T_1(\Gamma)$ at distance $d$ from a given vertex to grow like $(n^2p^4/2)^d$. This suggests the following behaviour may be likely:
\begin{enumerate}[(i)]
	\item for $p= \sqrt{f(n)/n}$ with $(2+\varepsilon)\log n< f(n)$  and $\log (f(n))=o(\log n)$, the diameter of $T_1(\Gamma)$ has a.a.s.\ size  $O(\log n / \log f)$;
	\item for every integer $k \geq 2$ and $n^{-\frac{1}{2}+\frac{1}{2k} }\ll p \ll n^{-\frac{1}{2}+\frac{1}{2(k-1)}}$, the diameter of $T_1(\Gamma)$ is a.a.s.\  exactly equal to $k$;
	\item for every integer $k \geq 3$ and $p=\theta(n^{-\frac{1}{2}+\frac{1}{2k} })$, the diameter of $T_1(\Gamma)$ is a.a.s.\  equal to $k$ or $k-1$.
\end{enumerate}
\noindent (Note that Proposition~\ref{prop: large p connectivity} establishes the behaviour in part (ii) above  in the case $k= 2$.)

\subsection*{Acknowledgements}
The initial discussions that led to this paper were carried out when the
first author visited the second and third author in Ume{\aa} in
January 2023 at the occasion of the third Midwinter Meeting in
Discrete Probability.  The hospitality of the Ume{\aa} mathematics
department and the financial support of the Simons Foundation 
are gratefully acknowledged. In
addition the research of the second
and third authors is supported by the Swedish Research Council grant
VR 2021-03687.

\bibliographystyle{plain}

\end{document}